\pdfoutput=1
\RequirePackage{silence}
\documentclass[a4paper,11pt]{amsart}
\usepackage[hmarginratio={1:1},vmarginratio={1:1},lmargin=70.0pt,tmargin=70.0pt]{geometry}
\usepackage{lmodern}

\usepackage{ifdraft}
\ifdraft{\usepackage[draft]{showkeys}}{\usepackage[final]{showkeys}}
\usepackage{calligra}
\usepackage[T1]{fontenc}
\usepackage[ugly]{nicefrac}
\usepackage{mathtools}
\mathtoolsset{mathic}
\usepackage[final]{microtype}
\usepackage{multicol}
\usepackage{tabularx,tabu}
\usepackage{latexsym,exscale,enumitem,amsfonts,amssymb,mathtools}
\usepackage{amsmath,amsthm,amsfonts,amssymb,amscd,textcomp,bbm}
\usepackage{stmaryrd}
\SetSymbolFont{stmry}{bold}{U}{stmry}{m}{n}
\usepackage[normalem]{ulem}
\usepackage{thmtools}
\usepackage{fancybox}
\usepackage[table]{xcolor}
\usepackage{mathrsfs}
\DeclareMathAlphabet{\mathscrbf}{OMS}{mdugm}{b}{n}
\DeclareFontEncoding{LS1}{}{}
\DeclareFontSubstitution{LS1}{stix}{m}{n}
\DeclareMathAlphabet{\mathpzc}{LS1}{stixscr}{m}{n}

\usepackage{caption} 
\definecolor{mygray}{gray}{0.6}
\definecolor{mygraydark}{gray}{0.4}
\definecolor{mygraylight}{gray}{0.8}

\definecolor{cherry}{RGB}{222,49,99}
\definecolor{cream}{RGB}{255,253,208}
\definecolor{corn}{RGB}{251,236,93}
\definecolor{citron}{RGB}{190,180,90}

\definecolor{spinach}{RGB}{46,139,87}
\definecolor{tomato}{RGB}{255,99,71}
\definecolor{pumpkin}{RGB}{224,180,80}

\definecolor{orchid}{RGB}{143,40,194}
\definecolor{lava}{RGB}{207,16,32}
\definecolor{mydarkblue}{RGB}{10,10,150}

\definecolor{myorange}{RGB}{225,127,0}
\definecolor{mygreen}{RGB}{0,225,0}
\definecolor{mypurple}{RGB}{128,0,128}
\definecolor{myred}{RGB}{255,0,0}
\definecolor{myblue}{RGB}{0,0,195}
\definecolor{myyellow}{RGB}{210,210,0}

\usepackage[all]{xy}
\SelectTips{cm}{}

\usepackage{tikz}
\tikzstyle{densely dotted}=[dash pattern=on \pgflinewidth off .5pt]
\tikzset{anchorbase/.style={baseline={([yshift=-0.5ex]current bounding box.center)}},
tinynodes/.style={font=\tiny, text height=0.25ex, text depth=0.05ex},
smallnodes/.style={font=\scriptsize, text height=0.75ex, text depth=0.15ex},
ssoergel/.style={line width=1.0,color=spinach},
tsoergel/.style={line width=1.0,color=tomato,densely dotted},
usual/.style={line width=1.0,color=black},
JW/.style={line width=1.0,densely dotted,color=black},
pQJW/.style={line width=1.0,densely dashed,color=black,fill=corn!60},
pJW/.style={line width=1.0,color=black,fill=orchid!70},
}
\usetikzlibrary{cd}
\usetikzlibrary{decorations}
\usetikzlibrary{decorations.markings}
\usetikzlibrary{decorations.pathreplacing}
\usetikzlibrary{decorations.pathmorphing}
\usetikzlibrary{arrows.meta,shapes,positioning,matrix,calc}
\usetikzlibrary{shapes.callouts}
\tikzstyle directed=[postaction={decorate,decoration={markings,
mark=at position #1 with {\arrow[line width=0.15mm, black]{>}}}}]
\tikzstyle odirected=[postaction={decorate,decoration={markings,
mark=at position #1 with {\arrow[line width=0.15mm, myorange]{>}}}}]
\tikzstyle smarked=[postaction={decorate,decoration={markings,
mark=at position #1 with {\fill[spinach] (0,0) circle (.065cm);}}}]
\tikzstyle tmarked=[postaction={decorate,decoration={markings,
mark=at position #1 with {\fill[tomato] (0,0) circle (.065cm);}}}]
\tikzstyle wmarked=[postaction={decorate,decoration={markings,
mark=at position #1 with {\fill[white] (0,0) circle (.1cm);}}}]

\allowdisplaybreaks

\newcommand{\ie}{\textsl{i.e.}}
\newcommand{\eg}{\textsl{e.g.}}
\newcommand{\cf}{\textsl{cf.}}
\newcommand{\etc}{\textsl{etc.}}

\newcommand{\muta}{\textsl{mutatis mutandis}}

\renewcommand{\dots}{\text{...}}
\renewcommand{\vdots}{\rotatebox{90}{\text{...}}}

\newcommand{\C}{\mathbb{C}}
\newcommand{\Z}{\mathbb{Z}}

\newcommand{\K}{\mathbb{K}}
\newcommand{\kk}{\mathbbm{k}}
\newcommand{\N}[1][]{\mathbb{N}_{#1}}

\newcommand{\placeholder}{{}_{-}}

\newcommand{\acts}{\centerdot}

\newcommand{\setstuff}[1]{\mathrm{#1}}
\newcommand{\catstuff}[1]{\mathbf{#1}}

\newcommand{\obstuff}[1]{\mathtt{#1}}
\newcommand{\morstuff}[1]{\mathrm{#1}}

\newcommand{\End}{\setstuff{End}}
\newcommand{\Hom}{\setstuff{Hom}}

\newcommand{\aZ}{Z}
\newcommand{\cZ}{\catstuff{Z}}

\newcommand{\prmod}[1]{p\catstuff{Mod}\text{-}{#1}}

\newcommand{\ppar}{\mathsf{p}}
\newcommand{\F}[1][\ppar]{\mathbb{F}_{\ppar}}
\newcommand{\pbase}[2]{[#1]_{#2}}
\newcommand{\eve}{\setstuff{Eve}}
\newcommand{\block}[2]{(#1)_{#2}}

\newcommand{\zigzag}{\obstuff{Z}}
\newcommand{\zigzagc}{\overline{\obstuff{Z}}}
\newcommand{\zigzagy}{\obstuff{Z}^{\prime}}
\newcommand{\idemy}[1][v]{e_{#1}}

\newcommand{\SLtwo}{\mathrm{SL}_{2}(\K)}
\newcommand{\tilt}{\catstuff{Tilt}}

\newcommand{\dist}{\mathsf{d}}

\newcommand{\pjw}[1][v{-}1]{\obstuff{e}_{#1}}

\newcommand{\Up}[1]{\mathrm{U}_{#1}}

\newcommand{\Down}[1]{\mathrm{D}_{#1}}

\newcommand{\loopy}[1]{\mathrm{L}_{#1}}
\newcommand{\loopyy}[1]{\mathrm{L}_{#1}^{\prime}}
\newcommand{\loopdown}[2]{\mathrm{L}^{#1}_{#2}}

\newcommand{\funcf}{\obstuff{f}}
\newcommand{\funcg}{\obstuff{g}}

\newcommand{\funcF}[1][S]{\obstuff{f}_{#1}}
\newcommand{\funcG}[1][S]{\obstuff{g}_{#1}}
\newcommand{\funcH}[1][S]{\obstuff{h}_{#1}}
\newcommand{\hull}[1][S]{\overline{#1}}

\newcommand{\Dset}[1][v]{\setstuff{D}(#1)}
\newcommand{\Cset}[1][v]{(#1)_{\ppar}^{\odot}}

\newcommand{\gtmod}[1]{{G_{#1}T}\text{-}p\catstuff{Mod}}

\newtheorem{theoremm}{Theorem}[section]

\declaretheoremstyle[
headfont=\bfseries, 
notebraces={[}{]},
bodyfont=\normalfont\itshape,
headpunct={},
postheadspace=1em,
spacebelow=10pt,
spaceabove=10pt, 
]{ourtheo}

\declaretheoremstyle[
headfont=\normalfont\bfseries,
notefont=\mdseries,
notebraces={(}{)},
bodyfont=\normalfont\slshape,
headpunct={},
postheadspace=1em,
spacebelow=10pt,
spaceabove=10pt, 
]{ourdef}

\declaretheorem[style=ourtheo,name=Theorem,numberlike=theoremm]{theorem}
\declaretheorem[style=ourtheo,name=Lemma,numberlike=theoremm]{lemma}

\declaretheorem[style=ourtheo,name=Theorem,qed=$\square$,numberlike=theoremm]{theoremqed}
\declaretheorem[style=ourtheo,name=Lemma,qed=$\square$,numberlike=theoremm]{lemmaqed}
\declaretheorem[style=ourtheo,name=Proposition,qed=$\square$,numberlike=theoremm]{propositionqed}
\declaretheorem[style=ourtheo,name=Corollary,qed=$\square$,numberlike=theoremm]{corollary}

\declaretheorem[style=ourdef,name=Definition,numberlike=theorem]{definition}
\declaretheorem[style=ourdef,name=Example,numberlike=theorem]{example}
\declaretheorem[style=ourdef,name=Remark,numberlike=theorem]{remark}

\declaretheorem[style=ourtheo,name=Theorem]{introtheorem}

\allowdisplaybreaks

\numberwithin{equation}{section}
\usepackage[hypertexnames=false]{hyperref}
\usepackage{cleveref,bookmark}
\hypersetup{
pdftoolbar=true,        
pdfmenubar=true,        
pdffitwindow=false,     
pdfstartview={FitH},    
pdftitle={The center of \texorpdfstring{$\mathrm{SL}_{2}$}{SL2} tilting modules},    
pdfauthor={Daniel Tubbenhauer and Paul Wedrich},     
pdfsubject={},   
pdfcreator={Daniel Tubbenhauer and Paul Wedrich},   
pdfproducer={Daniel Tubbenhauer and Paul Wedrich}, 
pdfkeywords={}, 
pdfnewwindow=true,      
colorlinks=true,       
linkcolor=mydarkblue,          
citecolor=teal,        
filecolor=magenta,      
urlcolor=orchid,          
linkbordercolor=lava,
citebordercolor=teal,
urlbordercolor=orchid,  
linktocpage=true
}

\renewcommand{\theequation}{\thesection-\arabic{equation}}
\let\fullref\autoref
\def\makeautorefname#1#2{\expandafter\def\csname#1autorefname\endcsname{#2}}
\makeautorefname{equation}{Equation}%
\makeautorefname{footnote}{footnote}%
\makeautorefname{item}{item}%
\makeautorefname{figure}{Figure}%
\makeautorefname{table}{Table}%
\makeautorefname{part}{Part}%
\makeautorefname{appendix}{Appendix}%
\makeautorefname{chapter}{Chapter}%
\makeautorefname{section}{Section}%
\makeautorefname{subsection}{Section}%
\makeautorefname{subsubsection}{Section}%
\makeautorefname{paragraph}{Paragraph}%
\makeautorefname{subparagraph}{Paragraph}%
\makeautorefname{theorem}{Theorem}%
\makeautorefname{theo}{Theorem}%
\makeautorefname{thm}{Theorem}%
\makeautorefname{addendum}{Addendum}%
\makeautorefname{addend}{Addendum}%
\makeautorefname{add}{Addendum}%
\makeautorefname{maintheorem}{Main theorem}%
\makeautorefname{mainthm}{Main theorem}%
\makeautorefname{corollary}{Corollary}%
\makeautorefname{corol}{Corollary}%
\makeautorefname{coro}{Corollary}%
\makeautorefname{cor}{Corollary}%
\makeautorefname{lemma}{Lemma}%
\makeautorefname{lemm}{Lemma}%
\makeautorefname{lem}{Lemma}%
\makeautorefname{sublemma}{Sublemma}%
\makeautorefname{sublem}{Sublemma}%
\makeautorefname{subl}{Sublemma}%
\makeautorefname{proposition}{Proposition}%
\makeautorefname{proposit}{Proposition}%
\makeautorefname{propos}{Proposition}%
\makeautorefname{propo}{Proposition}%
\makeautorefname{prop}{Proposition}%
\makeautorefname{property}{Property}
\makeautorefname{proper}{Property}
\makeautorefname{scholium}{Scholium}%
\makeautorefname{step}{Step}%
\makeautorefname{conjecture}{Conjecture}%
\makeautorefname{conject}{Conjecture}%
\makeautorefname{conj}{Conjecture}%
\makeautorefname{question}{Question}
\makeautorefname{questn}{Question}
\makeautorefname{quest}{Question}
\makeautorefname{ques}{Question}
\makeautorefname{qn}{Question}
\makeautorefname{definition}{Definition}%
\makeautorefname{defin}{Definition}%
\makeautorefname{defi}{Definition}%
\makeautorefname{def}{Definition}%
\makeautorefname{dfn}{Definition}%
\makeautorefname{notation}{Notation}
\makeautorefname{nota}{Notation}
\makeautorefname{notn}{Notation}
\makeautorefname{remark}{Remark}%
\makeautorefname{rema}{Remark}%
\makeautorefname{rem}{Remark}%
\makeautorefname{rmk}{Remark}%
\makeautorefname{rk}{Remark}%
\makeautorefname{remarks}{Remarks}%
\makeautorefname{rems}{Remarks}%
\makeautorefname{rmks}{Remarks}%
\makeautorefname{rks}{Remarks}%
\makeautorefname{example}{Example}%
\makeautorefname{examp}{Example}%
\makeautorefname{exmp}{Example}%
\makeautorefname{exam}{Example}%
\makeautorefname{exa}{Example}%
\makeautorefname{algorithm}{Algorithm}%
\makeautorefname{algo}{Algorithm}%
\makeautorefname{alg}{Algorithm}%
\makeautorefname{axiom}{Axiom}%
\makeautorefname{axi}{Axiom}%
\makeautorefname{ax}{Axiom}%
\makeautorefname{case}{Case}%
\makeautorefname{claim}{Claim}%
\makeautorefname{clm}{Claim}%
\makeautorefname{assumption}{Assumption}%
\makeautorefname{assumpt}{Assumption}%
\makeautorefname{conclusion}{Conclusion}%
\makeautorefname{concl}{Conclusion}%
\makeautorefname{conc}{Conclusion}%
\makeautorefname{condition}{Condition}%
\makeautorefname{condit}{Condition}%
\makeautorefname{cond}{Condition}%
\makeautorefname{construction}{Construction}%
\makeautorefname{construct}{Construction}%
\makeautorefname{const}{Construction}%
\makeautorefname{cons}{Construction}%
\makeautorefname{criterion}{Criterion}%
\makeautorefname{criter}{Criterion}%
\makeautorefname{crit}{Criterion}%
\makeautorefname{exercise}{Exercise}%
\makeautorefname{exer}{Exercise}%
\makeautorefname{exe}{Exercise}%
\makeautorefname{problem}{Problem}%
\makeautorefname{problm}{Problem}%
\makeautorefname{probm}{Problem}%
\makeautorefname{prob}{Problem}%
\makeautorefname{solution}{Solution}%
\makeautorefname{soln}{Solution}%
\makeautorefname{sol}{Solution}%
\makeautorefname{summary}{Summary}%
\makeautorefname{summ}{Summary}%
\makeautorefname{sum}{Summary}%
\makeautorefname{operation}{Operation}%
\makeautorefname{oper}{Operation}%
\makeautorefname{observation}{Observation}%
\makeautorefname{observn}{Observation}%
\makeautorefname{obser}{Observation}%
\makeautorefname{obs}{Observation}%
\makeautorefname{ob}{Observation}%
\makeautorefname{convention}{Convention}%
\makeautorefname{convent}{Convention}%
\makeautorefname{conv}{Convention}%
\makeautorefname{cvn}{Convention}%
\makeautorefname{warning}{Warning}%
\makeautorefname{warn}{Warning}%
\makeautorefname{note}{Note}%
\makeautorefname{fact}{Fact}%
\makeautorefname{hope}{Expectation}%
\makeautorefname{hope2}{Conjecture}%

\newcommand{\nnfootnote}[1]{%
\begin{NoHyper}
\renewcommand\thefootnote{}\footnote{#1}%
\addtocounter{footnote}{-1}%
\end{NoHyper}
}

\begin{document}
\title[The center of \texorpdfstring{$\mathrm{SL}_{2}$}{SL2} tilting modules]
{The center of \texorpdfstring{$\mathrm{SL}_{2}$}{SL2} tilting modules}
\author[Daniel Tubbenhauer and Paul Wedrich]{Daniel Tubbenhauer and Paul Wedrich}

\address{D.T.: Institut f{\"u}r Mathematik, Universit{\"a}t Z{\"u}rich, 
Winterthurerstrasse 190, Campus Irchel, Office Y27J32, CH-8057 Z{\"u}rich, 
Switzerland, \href{www.dtubbenhauer.com}{www.dtubbenhauer.com}}
\email{daniel.tubbenhauer@math.uzh.ch}

\address{P.W.: Mathematical Sciences Research Institute,
17 Gauss Way, Berkeley, CA 94720, USA, \href{http://paul.wedrich.at}{paul.wedrich.at}}
\email{p.wedrich@gmail.com}


\nnfootnote{\textit{Mathematics Subject Classification 2010.} Primary: 20G05, 20C20; Secondary: 17B10.}
\nnfootnote{\textit{Keywords.} Modular representation theory, tilting modules, generators and relations, positive characteristic.}

\begin{abstract}
In this note we compute the centers of the categories of tilting modules for
$G=\mathrm{SL}_{2}$ in prime characteristic, of tilting modules for
the corresponding quantum group at a complex root of unity, and of projective $G_{g}T$-modules when $g=1,2$.
\end{abstract}

\maketitle

\renewcommand{\theequation}{\thesection-\arabic{equation}}

\addtocontents{toc}{\protect\setcounter{tocdepth}{1}}

\section{Introduction}\label{section:intro}

Let $\K$ denote an algebraically closed field and $\tilt=\tilt\big(\SLtwo\big)$ the additive 
$\K$-linear category of (left-)tilting modules for the algebraic group $\SLtwo$.

In this note we compute the (categorical) center $\cZ(\tilt)$
of $\tilt$, using the explicit description of the 
Ringel dual of $\SLtwo$ from \cite{TuWe-tilting}. In characteristic zero we have
$\cZ(\tilt)\cong\prod_{\N}\K$ 
since $\tilt$ is semisimple with 
simple objects indexed by $\N$. Hence, our main concern is the case of prime characteristic. 
Thus, for the duration, $\K$ is of characteristic $\ppar\geq 2$.

\begin{introtheorem}\label{theorem:main}
We have isomorphisms of $\K$-algebras
\begin{gather*}
\cZ(\tilt)
\cong
{\K}[X_{v}\mid v\in\N]
\Big/
\langle X_{v}X_{w}\mid v,w\in\N\rangle.
\end{gather*}
\end{introtheorem}

We will provide an explicit isomorphism in \fullref{theorem:center-category} and
the discussion following it. For a possible interpretation of the central elements $X_v$
via Donkin's tensor product theorem see \fullref{subsection:tensor}.

Finally, in \fullref{section:othercases} 
we compute the centers of the categories of tilting modules 
in the quantum group case, see \fullref{theorem:qgroup}, 
and of projective $G_{g}T$-modules for $g=1,2$, 
see \fullref{theorem:grt}, both for $\mathrm{SL}_{2}$.

\medskip

\noindent\textbf{Acknowledgments.}
We like to thank Geordie Williamson for bringing 
the question about the center of tilting modules to our attention, and for helpful exchanges of emails.
We also thank Henning Haahr Andersen
and Catharina Stroppel for valuable email discussions and comments, and a referee for comments on a draft of this paper.

D.T. was sponsored by a chair in a random hotel during this project. P.W. was
supported by the National Science Foundation under Grant No. DMS-1440140, while
in residence at the Mathematical Sciences Research Institute in Berkeley,
California, during the Spring 2020 semester. 

\section{Preliminaries}\label{section:basics}

Throughout, we fix a prime $\ppar\geq 2$.
The main figure to keep in mind is \fullref{figure:main} 
which illustrates the underlying graph of the quiver algebra 
$\zigzag$ in the cases $\ppar=3$, $\ppar=5$ and $\ppar=7$.

\begin{figure}[ht]
\includegraphics[scale=.265]{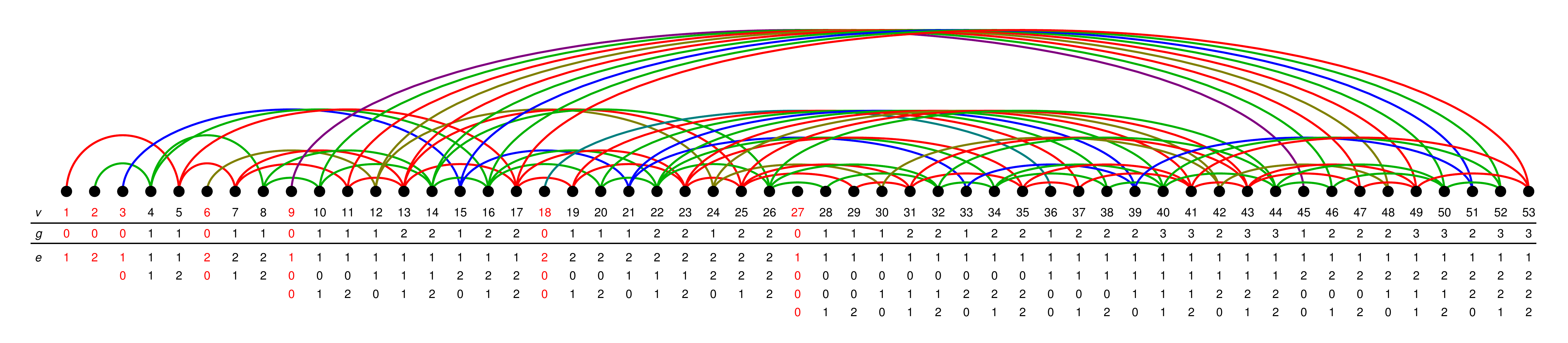}
\\
\includegraphics[scale=.265]{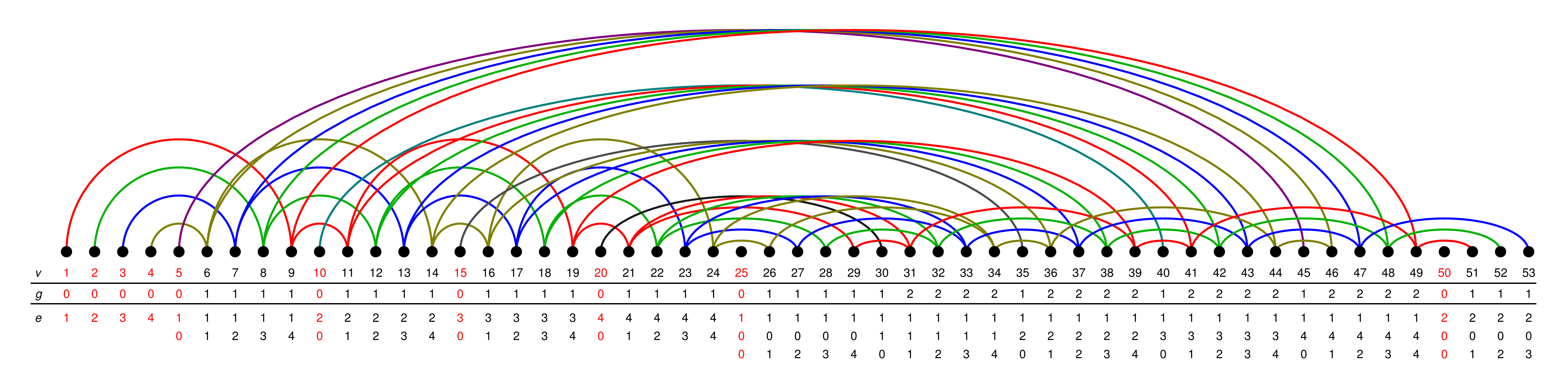}
\\
\includegraphics[scale=.265]{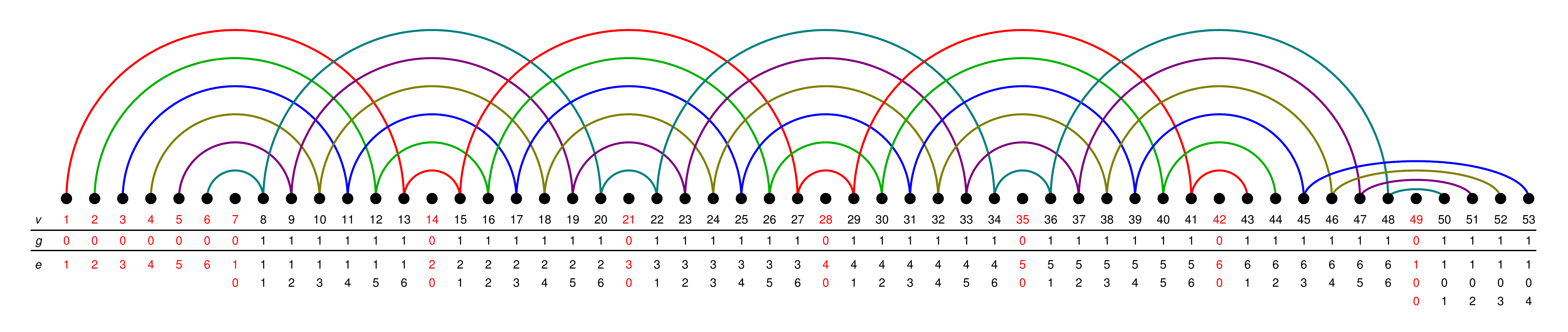}
\caption{The full subquivers containing the first $53$ vertices of the quiver
underlying $\zigzag_{\ppar}$ for $\ppar\in\{3,5,7\}$, showing from top to bottom
the numbers $v$ of the vertices, the generation of $v$ and the $\ppar$-adic
expansion of $v$.}
\label{figure:main}
\end{figure}

We also warn the reader that we will always have a $\rho$-shift of $1$ and the
crucial number will usually be ``the highest weight plus $1$'', which we will
denote by $v$, $w$ {\etc}

In the following, we review basic notation related to $\zigzag$. For
more details, we refer to \cite{TuWe-tilting}.

\subsection{Some combinatorics of \texorpdfstring{$\ppar$}{p}-adic expansions}\label{subsection:basics}

\begin{definition}
For any $v\in\N$ we write $\pbase{a_{j},\dots,a_{0}}{\ppar}:=\sum_{i=0}^{j}a_{i}
\ppar^{i}=v$ for the $\ppar$-adic expansion with digits $a_{i}\in
\{0,\dots,\ppar-1\}$ and $a_{j}\neq 0$. 
We sometimes also write $a_{h}$ for $h>j$, and then 
$a_{h}=0$, by convention.
Finally, 
the generation of $v$ is the number of non-zero digits of
$v$ minus $1$.
\end{definition}

\begin{remark} 
Conversely, we will sometimes specify numbers by $\ppar$-adic
expansions with negative digits such as $\pbase{3,-1,-6,-5,0,5,-6}{7}=320048$.
\end{remark}

The generation of $v$ is a complexity measure for the indecomposable
tilting module of highest weight $v-1$. For example the indecomposables of
generation zero are exactly the simple tilting modules. Moreover, to each
non-simple indecomposable tilting module one associates another tilting
module of generation one lower, called its mother, whose highest weight
(plus 1) is obtained by setting the lowest non-zero digit to zero. Tracing
the matrilinear ancestry of an indecomposable tilting module through
decreasing generation numbers, one arrives at a simple tilting module,
called an eve.

\begin{definition}
If $v=\pbase{a_{j},\dots,a_{0}}{\ppar}\in\N$ is of generation zero, then $v$ is
called an eve. The set of eves is denoted by $\eve$. 
\end{definition}

Note that $\eve=\eve^{<\ppar}\cup\eve^{\geq\ppar}$, with 
$\eve^{<\ppar}$ and $\eve^{\geq\ppar}$ having the evident meaning.

\begin{example}
For $\ppar=7$ we have 
$\eve^{<\ppar}=\{1,2,3,4,5,6\}$ 
and 
\begin{gather*}
\eve^{\geq\ppar}=\{\pbase{1,0}{7},\pbase{2,0}{7},\dots,\pbase{6,0}{7},\pbase{1,0,0}{7},\pbase{2,0,0}{7},\dots,\pbase{6,0,0}{7},\pbase{1,0,0,0}{7},\dots\}.
\end{gather*}
\end{example}

\begin{definition}
For two finite subsets $S,T\subset\N[0]$ the distance between them is defined as
$\dist(S,T)=\min\big\{|s-t|\mid s\in S,t\in T\big\}$, and we write $T>S$ to
indicate the requirement that every element in $T$ is strictly greater than
every element in $S$.
\end{definition}

\begin{definition}
For $S\subset\N[0]$ a finite set, we consider partitions $S=\bigsqcup_{i}S_{i}$ of
$S$ into subsets $S_{i}$ of consecutive integers, which we call stretches. We fix the coarsest such partition.

The set $S$ is called down-admissible for $v=\pbase{a_{j},\dots,a_{0}}{\ppar}$ if:
\begin{enumerate}[label=(\roman*)]

\setlength\itemsep{0.15cm}

\item $a_{\min(S_{i})}\neq 0$ for every $i$, and

\item if $s\in S$ and $a_{s+1}=0$, then $s+1\in S$.
\end{enumerate}
If $S\subset\N[0]$ is down-admissible for $v=\pbase{a_{j},\dots,a_{0}}{\ppar}$, then we define its 
downward reflection along $S$ as
\begin{gather*}
v[S]:=\pbase{a_{j},
\epsilon_{j-1}a_{j-1},\dots,\epsilon_{0}\,a_{0}}{\ppar},\quad
\epsilon_{k}
=
\begin{cases}
1 &\text{if }k\notin S,
\\
-1 &\text{if }k\in S.
\end{cases}
\end{gather*}

Conversely, $S$ is 
up-admissible for $v=\pbase{a_{j},\dots,a_{0}}{\ppar}$ if the following conditions are satisfied:
\begin{enumerate}[label=(\roman*)]

\setlength\itemsep{0.15cm}

\item $a_{\min(S_{i})}\neq 0$ for every $i$, and

\item if $s\in S$ and $a_{s{+}1}=\ppar-1$, then we also have $s+1\in S$.

\end{enumerate}
If $S\subset\N[0]$ is up-admissible for $v=\pbase{a_{j},\dots,a_{0}}{\ppar}$, then we define its 
upward reflection along $S$ as
\begin{gather*}
v(S):=
\pbase{a_{r(S)}^{\prime},\dots,a_{0}^{\prime}}{\ppar},
\quad
a_{k}^{\prime}= 
\begin{cases}
a_{k} &\text{if }k\notin S,k-1\notin S,
\\
a_{k}+2 &\text{if }k\notin S,k-1\in S,
\\
-a_{k} &\text{if }k\in S,
\end{cases}
\end{gather*}
where we extend the digits of $v$ by $a_{h}=0$ for $h>j$ if necessary, and
$r(S)$ is the biggest integer such that $a_{k}^{\prime}\neq 0$.
\end{definition}

\begin{example}
Note that $S=\{7,6\}$ is up-admissible for $v=\pbase{3,1,6,5,0,5,6}{7}$. To
compute $v(S)$ we let $v=\pbase{0,3,1,6,5,0,5,6}{7}$ and then we get
\begin{gather*}
v(S)=
\pbase{\uwave{0,3},1,6,5,0,5,6}{7}
=
\pbase{2,-3,1,6,5,0,5,6}{7}
=
\pbase{1,4,1,6,5,0,5,6}{7}.
\end{gather*}
The wave indicates the digits on which we apply $S$.
\end{example}

We tend to omit set brackets, {\eg}
for singleton sets $\{i\}$ we also write 
$v[i]$ and $v(i)$ instead of $v[\{i\}]$ and $v(\{i\})$.

\begin{definition}
If $S$ is up-admissible, then we denote by $\hull\subset\N[0]$ the
down-admissible hull of $S$, the smallest down-admissible set containing $S$, if it exists.
\end{definition}

\begin{example}\label{example:downad}
Let $\ppar=7$ and $v=\pbase{3,1,6,5,0,5,6}{7}$.

\begin{enumerate}[label=(\alph*)]

\setlength\itemsep{0.15cm}

\item The singleton sets $\{2\}$, $\{1\}$ and $\{i\}$ for $i\in\N[>5]$
are not down-admissible for $v$. Of these only $\{1\}$ has a down-admissible
hull and it is $\hull[{\{1\}}]=\{2,1\}$.

\item Hence, the singleton sets $\{5\}$, $\{4\}$, $\{3\}$ and 
$\{0\}$ are minimal down-admissible for $v$, 
and we have
\begin{gather*}
v[5]=\pbase{3,-1,6,5,0,5,6}{7}
,\quad
v[4]=\pbase{3,1,-6,5,0,5,6}{7}
,\\
v[3]=\pbase{3,1,6,-5,0,5,6}{7}
,\quad
v[0]=\pbase{3,1,6,5,0,5,-6}{7}.
\end{gather*}
\emph{Toghether with $\{2,1\}$, these are 
all minimal down-admissible sets for $v$}.

\item The singleton sets which are up-admissible 
for $v$ are $\{6\}$, $\{5\}$, $\{4\}$, $\{1\}$ and 
$\{0\}$.

\item The set $S=\{5,4,3|0\}$ is down- and up-admissible for $v$, and $v(S)$
and $v[S]$ can be illustrated via
\begin{gather*}
v[5,4,3|0]=\pbase{3,\underline{1,6,5},0,5,\underline{6\!\!\phantom{,}}}{7}=\pbase{3,-1,-6,-5,0,5,-6}{7},
\\
v(5,4,3|0)=\pbase{3,\uwave{1,6,5},0,5,\uwave{6}}{7}=\pbase{5,-1,-6,-5,0,7,-6}{7}.
\end{gather*}

\end{enumerate}
We have marked the digits where we reflect in the set $S$, either \underline{down} or \uwave{up}.
\end{example}

\subsection{The Ringel dual}\label{subsection:dual}

Let $\K=\overline{\K}$ denote a field of 
characteristic $\ppar$ with prime field $\F$.

\begin{definition}
We define two functions $\funcf,\funcg\colon\F\to\F$ via
\begin{gather*}
\funcf(a)=
\begin{cases} 
(-1)^{a}\tfrac{2}{a}
&\text{if }1\leq a\leq\ppar-2,
\\
0 
&\text{if }a=0\text{ or }a=\ppar-1,
\end{cases}
\quad
\funcg(a)=
\begin{cases}
-(\tfrac{a+1}{a})
&\text{if }1\leq a\leq\ppar-1,
\\
-2 &\text{if }a=0.
\end{cases}
\end{gather*}
\end{definition}

Note that $\funcf(\ppar-1)=\funcg(\ppar-1)=0$ and $\funcg(a)=\funcg(\ppar-a-1)^{-1}$ for $a\neq 0,\ppar-1$. 

Let $\tilt$ denote the category of 
(left) tilting modules of $\SLtwo$, see {\eg} \cite[Section 1]{Wi-algebraic-sheaves} for a concise summary of 
the main definitions and properties regarding $\tilt$.
Let $\obstuff{T}(v-1)$ denote (a choice of representative of) the indecomposable tilting module 
of highest weight $v-1$.

\begin{definition}
Define a $\K$-algebra
\begin{gather*}
\zigzag:=
{\textstyle\bigoplus_{v,w\in\N}} \Hom_{\tilt}\big(\obstuff{T}(v-1),\obstuff{T}(w-1)\big).
\end{gather*}
\end{definition}

Let $\pjw[v{-}1]$ be the idempotent in $\zigzag$ corresponding to $\obstuff{T}(v-1)$.

\begin{definition}
For each finite $S\subset\N[0]$ we define scaling 
operators $\funcF,\funcG,\funcH\in\zigzag$ on $v=\pbase{a_{j},\dots,a_{0}}{\ppar}$ as
\begin{gather*}
\funcF\pjw[v{-}1] 
=
\funcf(a_{\max(S)+1})\pjw[v{-}1], 
\quad
\funcG\pjw[v{-}1] 
= 
\funcg(a_{\max(S)+1})\pjw[v{-}1],
\quad
\funcH\pjw[v{-}1] 
=
\funcg(a_{\max(S)+1}-1)\pjw[v{-}1].
\end{gather*}
\end{definition}

\begin{example}
Let again $\ppar=7$, $v=\pbase{3,1,6,5,0,5,6}{7}$ and  $S=\{5,4,3|0\}$. Then
\begin{gather*}
\funcF\pjw[v{-}1] 
=
\funcf(3)\pjw[v{-}1]=4\pjw[v{-}1], 
\quad
\funcG\pjw[v{-}1] 
= 
\funcg(3)\pjw[v{-}1]=\pjw[v{-}1],
\quad
\funcH\pjw[v{-}1] 
=
\funcg(2)\pjw[v{-}1]=2\pjw[v{-}1].
\end{gather*}
\end{example}

The following identifies $\zigzag$ explicitly, and can be taken as an abstract
definition of a quiver algebra isomorphic to $\zigzag$, see also
\fullref{remark:quiver} below.

\begin{theoremqed}(See \cite[Theorem 3.2]{TuWe-tilting}.)\label{theorem:main-tl-section}
The algebra $\zigzag$ is generated by $\pjw[v{-}1]$ 
for $v\in\N$, 
and elements $\Down{S}\pjw[v{-}1]$ and $\Up{S^{\prime}}\pjw[v{-}1]$, 
where $S$ and $S^{\prime}$ denote minimal down- and up-admissible stretches for $v$, respectively.
These generators are subject to the following complete set of relations.
\begin{enumerate}[label=(\arabic*)]

\setlength\itemsep{0.15cm}

\item \emph{Idempotents.}
\begin{gather*}
\pjw[v{-}1]\pjw[w{-}1]=\delta_{v,w}\pjw[v{-}1],
\quad 
\pjw[{v[S]{-}1}]\Down{S}\pjw[v{-}1]
=
\Down{S}\pjw[v{-}1],
\quad
\pjw[v(S^{\prime}){-}1]\Up{S^{\prime}}\pjw[v{-}1]
=
\Up{S^{\prime}}\pjw[v{-}1].
\end{gather*}

\item \emph{Containment.}
If $S^{\prime} \subset S$, then we have
\begin{gather*}
\Down{S^{\prime}}\Down{S}\pjw[v{-}1]=0,
\quad
\Up{S}\Up{S^{\prime}}\pjw[v{-}1]=0.
\end{gather*}

\item \emph{Far-commutativity.}
If $\dist(S,S^{\prime})>1$, then 
\begin{gather*}
\Down{S}\Down{S^{\prime}}\pjw[v{-}1]=\Down{S^{\prime}}\Down{S}\pjw[v{-}1],
\quad
\Down{S}\Up{S^{\prime}}\pjw[v{-}1]=\Up{S^{\prime}}\Down{S}\pjw[v{-}1],
\quad
\Up{S}\Up{S^{\prime}}\pjw[v{-}1]=\Up{S^{\prime}}\Up{S}\pjw[v{-}1].
\end{gather*}

\item \emph{Adjacency relations.}
If $\dist(S,S^{\prime})=1$ and $S^{\prime}>S$, then
\begin{gather*}
\Down{S^{\prime}}\Up{S}\pjw[v{-}1]=\Down{S{\cup}S^{\prime}}\pjw[v{-}1],
\quad
\Down{S}\Up{S^{\prime}}\pjw[v{-}1]=\Up{S^{\prime}{\cup}S}\pjw[v{-}1],
\\
\Down{S^{\prime}}\Down{S}\pjw[v{-}1]=\Up{S}\Down{S^{\prime}}\funcH[S]\pjw[v{-}1],
\quad
\Up{S}\Up{S^{\prime}}\pjw[v{-}1]=\funcH[S]\Up{S^{\prime}}\Down{S}\pjw[v{-}1].
\end{gather*}

\item \emph{Overlap relations.}
If $S^{\prime}\geq S$ with $S^{\prime}\cap S=\{s\}$ and $S^{\prime}\not\subset S$, then we have
\begin{gather*}
\Down{S^{\prime}}\Down{S}\pjw[v{-}1]=\Up{\{s\}}\Down{S}\Down{S^{\prime}{\setminus}\{s\}}\pjw[v{-}1],
\quad
\Up{S}\Up{S^{\prime}}\pjw[v{-}1]=\Up{S^{\prime}{\setminus}\{s\}}\Up{S}\Down{\{s\}}\pjw[v{-}1].
\end{gather*}

\item \emph{Zigzag.}
\begin{gather*}
\Down{S}\Up{S}\pjw[v{-}1]=\Up{\hull[S]}\Down{\hull[S]}\funcG[S]\pjw[v{-}1] 
+\Up{T}\Up{\hull[S]}\Down{\hull[S]}\Down{T}\funcF[S]\pjw[v{-}1].
\end{gather*}
Here, if the down-admissible hull $\hull[S]$, or the smallest minimal
down-admissible stretch $T$ with $T>\hull[S]$ does not exist, then the involved
symbols are zero by definition.

\end{enumerate}

The elements of the form
\begin{gather*}
\tag{Basis}\qquad\pjw[w{-}1]\Up{S_{i_{l}}^{\prime}}
\cdots\Up{S_{i_{0}}^{\prime}}\Down{S_{i_{0}}}\cdots\Down{S_{i_{k}}}\pjw[v{-}1],
\end{gather*}
\hspace{1cm} with $S_{i_{l}}^{\prime}>\cdots>S_{i_{0}}^{\prime}$, and $S_{i_{0}}<\cdots<S_{i_{k}}$, 
form a basis for $\pjw[w{-}1]\zigzag\pjw[v{-}1]$. 

Finally, any word $\pjw[w{-}1]\morstuff{F}\pjw[v{-}1]$ in the generators 
of $\zigzag$ expands in \emph{(Basis)} via \emph{(1)--(6)}.
\end{theoremqed}

\begin{remark}\label{remark:quiver}
In fact, as already mentioned above, we could alternatively define $\zigzag$ as
the quiver algebra with underlying graphs as in \fullref{figure:main}, using
highest weights to index the vertices, and two arrows directed in opposite
directions for each edge in this graph. The elements $\pjw[v{-}1]$ are the
vertex idempotents of this quiver algebra, and the elements
$\Down{\placeholder}$, called down arrows, $\Up{\placeholder}$, called up
arrows, are the arrows in this quiver, pointing left and right, respectively.
\end{remark}

\begin{remark}
Note that all appearing scalars in the presentation of 
$\zigzag$ in \fullref{theorem:main-tl-section} are from $\F$ rather than $\K$. 
\end{remark}

\begin{remark}
In \fullref{theorem:main-tl-section}.(4) and (6), the right-hand sides of the shown relations feature 
morphisms indexed by admissible subsets that are not necessarily minimal. These morphisms are defined to be
\begin{gather}\label{eq:more-generators}
\Down{S}\pjw[v{-}1]
:=
\Down{S_{i_{1}}}\cdots \Down{S_{i_{k}}}\pjw[v{-}1],
\quad
\Up{S^{\prime}}\pjw[v{-}1]
:=
\Up{S^{\prime}_{i_{l}}}\cdots \Up{S^{\prime}_{i_{1}}}\pjw[v{-}1],
\end{gather}
where the products are taken over the minimal down- respectively up-admissible stretches $S_{i_{j}}$ and 
$S^{\prime}_{i_{j}}$, 
such that $S=\bigsqcup_{j} S_{i_{j}}$ and $S^{\prime}=\bigsqcup_{j} S^{\prime}_{i_{j}}$, with $S_{i_{1}}<\cdots<S_{i_{k}}$ and 
$S^{\prime}_{i_{l}}>\cdots>S^{\prime}_{i_{1}}$. 
\end{remark}

\begin{example}\label{example:quiver}
To be completely explicit with respect to \fullref{remark:quiver}: if $\ppar=3$,
then the quiver of the idempotent truncation of $\zigzag$ supported on the vertices
$0$, $4$, $6$, $10$, $12$ and $16$ would be 
\begin{gather*}
\begin{tikzcd}[ampersand replacement=\&,column sep=3em]  
0
\ar[r,"\Up{\{0\}}",black,yshift=0.1cm]
\&
4
\ar[r,"\Up{\{0\}}",black,yshift=0.1cm]
\ar[l,"\Down{\{0\}}",orchid,yshift=-0.1cm]
\ar[rrrr,"\Up{\{1\}}",spinach,yshift=0.45cm,bend left=30]
\&
6
\ar[r,"\Up{\{0\}}",black,yshift=0.1cm]
\ar[l,"\Down{\{0\}}",orchid,yshift=-0.1cm]
\ar[rr,"\Up{\{1\}}",spinach,yshift=-0.25cm,bend right=30]
\&
10
\ar[r,"\Up{\{0\}}",black,yshift=0.1cm]
\ar[l,"\Down{\{0\}}",orchid,yshift=-0.1cm]
\&
12
\ar[r,"\Up{\{0\}}",black,yshift=0.1cm]
\ar[l,"\Down{\{0\}}",orchid,yshift=-0.1cm]
\ar[ll,"\Down{\{1\}}",tomato,yshift=-0.45cm,bend left=30]
\&
16
\ar[l,"\Down{\{0\}}",orchid,yshift=-0.1cm]
\ar[llll,"\Down{\{1\}}",tomato,yshift=0.25cm,bend right=30]
\end{tikzcd}
.
\end{gather*}
(The colors, here and below, are only to ease readability.)
Here the labeling is coming from {\eg}
\begin{gather*}
17=\pbase{1,2,2}{3},
\quad
17[0]=\pbase{1,2,\underline{2}}{3}
=\pbase{1,2,-2}{3}=13,
\quad
17[1]=\pbase{1,\underline{2},2}{3}
=\pbase{1,-2,2}{3}=5,
\end{gather*}
which give the two downward arrows from $16$, 
namely $\Down{\{0\}}\pjw[16]=\pjw[12]\Down{\{0\}}$ 
and $\Down{\{1\}}\pjw[16]=\pjw[4]\Down{\{1\}}$.
We also have 
$\Down{\{1,0\}}\pjw[16]=\Down{\{0\}}\Down{\{1\}}\pjw[16]$.
\end{example}

\begin{definition}
For any down-admissible set $S$ for $v$ we define the loop
\begin{gather*}
\loopdown{S}{v{-}1}:=\Up{S}\Down{S}\pjw[v{-}1].
\end{gather*}
\end{definition}

\begin{lemmaqed}(See \cite[Lemma 3.23]{TuWe-tilting}.)
\label{lemma:dualnumbers} 
Let $v\in\N$ with minimal 
down-admissible stretches $S_{j},\dots,S_{0}$. Then we have the $\K$-algebra isomorphism
\begin{gather*}
\End_{\tilt}\big(\obstuff{T}(v-1)\big)
\cong
\K\big[\loopdown{S_{j}}{v{-}1},\dots,\loopdown{S_{0}}{v{-}1}\big]
\Big/\big\langle(\loopdown{S_{j}}{v{-}1})^{2},\dots, (\loopdown{S_{0}}{v{-}1})^{2}\big\rangle,
\end{gather*}
and if $S$ is down-admissible for $v$, then $\loopdown{S}{v{-}1}=\prod_{k\mid S_{k}{\subset}S}\loopdown{S_{k}}{v{-}1}$.
\end{lemmaqed}

\begin{example}
Again, let us consider $\ppar=7$ and $v=\pbase{3,1,6,5,0,5,6}{7}$. Recall that
we have calculated the minimal down-admissible stretches of $v$ in
\fullref{example:downad}.(b). Hence, $\End_{\tilt}\big(\obstuff{T}(v-1)\big)$
has generators $\loopdown{\{5\}}{v{-}1}$, $\loopdown{\{4\}}{v{-}1}$,
$\loopdown{\{3\}}{v{-}1}$ and $\loopdown{\{0\}}{v{-}1}$ as well as
$\loopdown{\{2,1\}}{v{-}1}$. The maximal loop
$\loopdown{\{5|4|3|2,1|0\}}{v{-}1}=\loopdown{\{5\}}{v{-}1}\loopdown{\{4\}}{v{-}1}
\loopdown{\{3\}}{v{-}1} \loopdown{\{2,1\}}{v{-}1} \loopdown{\{0\}}{v{-}1}$
can be thought of as a head-to-socle map on $\obstuff{T}(v-1)$.
\end{example}

For later use we also recall:

\begin{lemmaqed}(See \cite[Lemma 3.23]{TuWe-tilting}.)
\label{lemma:dualnumbers2} 
We have
\begin{gather}\label{eq:DUD}
\Down{S}\Up{S}\Down{S}\pjw[v{-}1]=0, 
\quad
\pjw[v{-}1]\Up{S}\Down{S}\Up{S}=0,
\end{gather}
whenever $S$ is down-admissible for $v$.
\end{lemmaqed}

\subsection{Closures of the algebra}\label{subsection:closure}

Using \fullref{theorem:main-tl-section} we get a sequence
\begin{gather*}
\zigzag^{1}
\twoheadleftarrow
\zigzag^{2}
\twoheadleftarrow
\zigzag^{3}
\twoheadleftarrow
\dots
\twoheadleftarrow
\lim_{\longleftarrow}
\zigzag^{i}
=:
\zigzagc,
\end{gather*}
where $\zigzag^{i}$ is the quotient of 
$\zigzag$ obtained by the ideal generated by 
$\{\pjw[v{-}1]\mid v>i\}$. (Note that the category of $\K$-algebras
is complete, {\ie} has all limits, so $\zigzagc$ is indeed a $\K$-algebra.

\begin{remark}
Note that elements in $\zigzag$ are finite $\K$-linear combinations of elements from \fullref{theorem:main-tl-section}.(Basis), while 
elements in $\zigzagc$ can be (countably) infinite 
$\K$-linear combinations of these. 
In particular, the unit of $\zigzagc$ is
\begin{gather}\label{eq:unit}
1={\textstyle\sum_{v\in\N}}\,\pjw[v{-}1]\in\zigzagc,
\end{gather}
while $\zigzag$ is only locally unital.
\end{remark}

Below $\zigzagy$ will denote either $\zigzag$ 
or $\zigzagc$.

\section{The center of the quiver algebra}\label{section:center-alg}

We will now compute the center of $\zigzagy$.

\subsection{Reduction to connected components}\label{subsection:reduction-center}

Note that \fullref{theorem:main-tl-section} implies that the $\K$-algebra $\zigzagc$ decomposes as
\begin{align*}
\zigzag
&= 
{\textstyle\bigoplus_{e\in\eve}}\,\zigzag_{e{-}1}
,\quad
\zigzag_{e{-}1}
:={\textstyle\bigoplus_{v,w\in\block{e}{\ppar}}}\,\idemy[w{-}1]\zigzag\idemy[v{-}1]
,\\
\zigzagc
&= 
{\textstyle\prod_{e\in\eve}}\,\zigzagc_{e{-}1}
,\quad
\zigzagc_{e{-}1}
:={\textstyle\prod_{v,w\in\block{e}{\ppar}}}\,\idemy[w{-}1]\zigzagc\idemy[v{-}1].
\end{align*}
Here $\block{e}{\ppar}$ denotes the set of natural numbers $v$ such that the vertex
$v-1$ is in the connected component of the graph underlying $\zigzag$
which contains $e-1$. 

\begin{example}
As can be seen in \fullref{figure:main}, $\block{1}{3}=\{1,5,7,11,13,17,\dots\}$,
see also \fullref{example:quiver}.
\end{example}

Letting $\aZ(\placeholder)$ 
denote the center of an algebra, the following is thus immediate.

\begin{lemmaqed}\label{lemma:center-reduction}
We have $\aZ(\zigzag)=\bigoplus_{e\in\eve}
\aZ(\zigzag_{e{-}1})$ and $\aZ(\zigzagc)=\prod_{e\in\eve}
\aZ(\zigzagc_{e{-}1})$
\end{lemmaqed}

The presentation of $\zigzagy$ also immediately gives:

\begin{lemmaqed}(See \cite[Proposition 5.3]{TuWe-tilting}.)
\label{lemma:isoalgs}
There are isomorphisms of algebras
$\zigzagy_{e{-}1}\cong\zigzagy_{e^{\prime}{-}1}$ for all $e,e^{\prime}\in\eve$
with equal non-zero digits. 
\end{lemmaqed}

\begin{remark}
In fact $\zigzagy_{e{-}1}\cong\zigzagy_{e^{\prime}{-}1}$ for all $e,e^{\prime}\in\eve$, regardless of the 
digits. But the isomorphism is, in contrast to the one from 
\fullref{lemma:isoalgs}, not immediate from \fullref{theorem:main}, the reason being the scalars 
$\funcF$, $\funcG$ and $\funcH$ appearing therein.
\end{remark}

Thus, it suffices to compute $\aZ(\zigzagy_{e{-}1})$ for
$e\in\eve^{<\ppar}\subset\eve$.

\subsection{The central elements}\label{subsection:elements-center}

\begin{definition}\label{definition:some-elements}
Let $v=\pbase{a_{j},\dots,a_{0}}{\ppar}\in\N$. 
For $i\neq j$ and $a_{i}\neq 0$, we define elements in 
$\zigzagy$ by
\begin{gather*}
\Down{i}\pjw[v{-}1] 
:= 
\Down{\hull[\{i\}]}\pjw[v{-}1],
\quad 
\pjw[v{-}1]\Up{i}
:=
\pjw[v{-}1]\Up{\hull[\{i\}]},
\quad
\loopdown{}{i}\pjw[v{-}1] 
:=
(-1)^{a_{i}}a_{i}\Up{i}\Down{i}\pjw[v{-}1].
\end{gather*}
Furthermore, we define
\begin{gather*}
\Dset:=\{i\in\N[0]\mid i<j,a_{i}\neq 0\}.
\end{gather*}
In words, $\Dset$ is the set of non-zero, non-leading digits of 
$v=\pbase{a_{j},\dots,a_{0}}{\ppar}$.
For $\setstuff{S}\subset\Dset$ we consider
\begin{gather}\label{eq:commute}
\loopdown{}{\setstuff{S}}\pjw[v{-}1]:=
{\textstyle\prod_{i\in\setstuff{S}}}\,
\loopdown{}{i}\pjw[v{-}1]. 
\end{gather}
\end{definition}

Note that the factors in \eqref{eq:commute} commute by \fullref{lemma:dualnumbers}.

\begin{example}
For $\ppar=7$ and $v=\pbase{3,1,6,5,0,5,6}{7}$ 
we have $\Dset=\{0,1,3,4,5\}$. Thus, for $\setstuff{S}=\Dset$ we have 
$\loopdown{}{\setstuff{S}}\pjw[v{-}1]=\loopdown{}{5}\loopdown{}{4}\loopdown{}{3}\loopdown{}{1}\loopdown{}{0}\pjw[v{-}1]$.
\end{example}

\begin{definition}\label{definition:equivrel}
Let $\setstuff{S}\subset\N[0]$ be finite. 
We define an equivalence relation $\sim_{\setstuff{S}}$ on $\N$
as follows. First, we let $v\sim_{\setstuff{S}}v$ for all $v\in\N$.
Further, for $v,w\in\N$ with $v\neq w$ we declare $v\sim_{\setstuff{S}}w$ if $\setstuff{S}\subset\Dset\cap
\Dset[w]$ and $v=w[k]$ or $w=v[k]$ for some $k\notin\setstuff{S}$,
Finally, we take the transitive closure. 
\end{definition}

In words, we have $v\sim_{\setstuff{S}}w$ if and 
only if either $v=w$, or $v-1$ and $w-1$
are connected by a chain of arrows 
$\Down{k}$ or $\Up{k}$ for $k\notin\setstuff{S}$ such that all
vertices that are involved have $\setstuff{S}$ among their non-maximal 
and non-zero digits. We
call such a chain an $\setstuff{S}$-path.

\begin{definition} 
For $v\in\block{e}{\ppar}$ we introduce notation for the
$\sim_{\Dset}$ equivalence class of $v$: 
\begin{gather*}
\Cset:=\{
w\in\block{e}{\ppar}\mid v\sim_{\Dset}w
\}.
\end{gather*}
\end{definition}

\begin{example}\label{example:cset}
Take again $\ppar=7$, $v=\pbase{3,1,6,5,0,5,6}{7}$ 
and $\setstuff{S}=\Dset=\{0,1,3,4,5\}$. Then:
\begin{enumerate}[label=(\alph*)]

\setlength\itemsep{0.15cm}

\item We have $w=v(5,4,3)=\pbase{4,5,0,2,0,5,6}{7}\notin\Cset$, 
since $4\in\Dset$, but $4\notin\Dset[w]$. 

\item Recall that the only up-admissible singleton set not in $\Dset$ is
$\{6\}$, see \fullref{example:downad}.(c). Thus, the only direct neighbor of $v$
in $\Cset$ is $v_{1}=v(6)=\pbase{1,4,1,6,5,0,5,6}{7}$.

\item The set $\Cset$ is an infinite set since 
$v_{1}=v(6)=\pbase{1,4,1,6,5,0,5,6}{7}\in\Cset$, and recursively $v_{k}=v_{k{-}1}(5+k)\in\Cset$.
Note that the first six digits of the $v_{k}$ agree with those of $v$.

\end{enumerate} 
\end{example}

\begin{lemma}\label{lemma:charCv} 
Let $v=\pbase{a_{j},\dots,a_{0}}{\ppar}\in\N$ and
$w=\pbase{b_{k},\dots,b_{0}}{\ppar}\in\N$. If $w\in
\Cset$, then $b_{i}=a_{i}$ for $1\leq i<j$.
\end{lemma}

\begin{proof} 
We prove by induction on 
the length $\ell$ of a $\Dset$-path that the
first $j$ digits of all elements $w\in\Cset$ agree 
with those of $v$. For
$\ell=0$ there is nothing to 
prove. Now suppose $w$ is at the end of a
$\Dset$-path of length $\ell\geq 1$, 
whose last step is a morphism $\Up{k}$ or
$\Down{k}$ from $w^{\prime}$ to $w$, 
with $k\notin\Dset$. Since $w^{\prime}$
is reached from $v$ by a $\Dset$-path 
of length $\ell-1$, the induction hypothesis
implies that the first $j$ digits of 
$w^{\prime}$ agree with those of $v$. In
particular, each of the first $j$ 
digits is either in $\Dset$ or zero, and
so $k\geq j$. Then $w=w^{\prime}[k]$, respectively $w=w^{\prime}(k)$, have the same first
$j$ digits as $w^{\prime}$ and also as $v$.
\end{proof}

\fullref{lemma:charCv} implies that $\Cset$ for
$v=\pbase{a_{j},\dots,a_{0}}{\ppar}\in\N$ is a shifted copy of the block
$(a_{j})_{\ppar}$, see \cite[Proposition 5.2]{TuWe-tilting}. As a consequence we
get the following.

\begin{corollary}\label{corollary:Cv} 
For every $v\in \N$, the 
set $\Cset$ is infinite, and if $w\in\Cset$, 
then $w\geq v$.
\end{corollary}

The loops that we define now are only elements 
in $\zigzagc$ but not in $\zigzag$.

\begin{definition}\label{definition:Lv}
For $v\in\block{e}{\ppar}$ and $v\neq e$ we define elements 
\begin{gather}\label{eq:maxloop}
\loopdown{}{v}:=
{\textstyle\sum_{w\in\Cset}}\,
\loopdown{}{\Dset}\pjw[w{-}1]
\in\zigzagc_{e{-}1}
.
\end{gather} 
\end{definition}

In words, $\loopdown{}{v}$ consists of the maximal loop at $v-1$, together with the
sum of all loops of the same type on $w-1$ for all $w$ in the shifted block $\Cset$.

\subsection{The center}\label{subsection:alg-center}

By \fullref{lemma:center-reduction}, the following computes 
the center $\aZ(\zigzagy)$:

\begin{theorem}\label{theorem:center-algebra}
For $e\in\eve^{<\ppar}$ we have algebra isomorphisms
\begin{gather*}
\aZ(\zigzag_{e{-}1})
\cong\{0\},
\quad
\aZ(\zigzagc_{e{-}1})
\xrightarrow{\cong}
{\K}\big[X_{v}\mid v\in\block{e}{\ppar}\setminus\{e\}\big]
\Big/
\big\langle X_{v}X_{w}\mid v,w\in\block{e}{\ppar}\setminus\{e\}\big\rangle
,\quad 
\loopdown{}{v}\mapsto X_{v}.
\end{gather*}
\end{theorem}

\begin{proof}

We first observe that we have $\aZ(\zigzagy_{e{-}1})\subset
\prod_{v\in\block{e}{\ppar}}\pjw[v{-}1]\zigzagy\pjw[v{-}1]$. Indeed, if $z\in
\aZ(\zigzagy_{e{-}1})$ and $v\in\block{e}{\ppar}$, then we have $z \pjw[v{-}1]=z
\pjw[v{-}1]^{2}=\pjw[v{-}1]z\pjw[v{-}1]$. The observation follows since $z$
is by assumption a finite or infinite sum of terms $z \pjw[v{-}1]$ for $v\in\block{e}{\ppar}$.

\begin{enumerate}[label=$\bullet$]
\setlength\itemsep{0.15cm}

\item Let us first consider $\zigzagc_{e{-}1}$. Here \fullref{lemma:loopscentral}.(b) (proven below) shows that the infinite sums
$\loopdown{}{v}$ defined above are in the center $\aZ(\zigzagc_{e{-}1})$, while
\fullref{lemma:loopscentral-op} (also proven below) shows that they, together with the identity, give a basis of
$\aZ(\zigzagc_{e{-}1})$. Thus,
\begin{gather*}
\aZ(\zigzagc_{e{-}1})\cong
{\K}\big\langle 1,\loopdown{}{v}\mid v\in\block{e}{\ppar}\setminus\{e\}\big\rangle.
\end{gather*}
The second isomorphism then follows because we already know the relations among
the $\loopdown{}{v}$, {\cf} \fullref{lemma:prodzero}.

\item For the first isomorphism, note that we have an inclusion
$\aZ(\zigzag_{e{-}1})\hookrightarrow\aZ(\zigzagc_{e{-}1})$ of non-unital
$\K$-algebras. However, \fullref{lemma:loopscentral} implies that all
non-trivial elements of $\aZ(\zigzagc_{e{-}1})$ are supported on infinitely many
idempotents, so $\aZ(\zigzag_{e{-}1})=\aZ(\zigzagc_{e{-}1})\cap\zigzag_{e{-}1} =\{0\}$.\qedhere
\end{enumerate}

\end{proof}

\subsection{Some lemmas for the proof of \fullref{theorem:center-algebra}}\label{subsection:lemmas}

Next, we identify the relations among the $\loopdown{}{v}$.

\begin{lemma}\label{lemma:zerodigit}
For $e\in\eve^{<\ppar}$ and $v\in\block{e}{\ppar}$ with $v\neq e$ we have $0\in\Dset$.
\end{lemma}

\begin{proof}
Any $v\in \block{e}{\ppar}$ can be reached from $e$ through a finite sequence of
$\Down{i}$ and $\Up{i}$. It is straightforward to check that under each such
an arrow the zeroth digit stays unchanged or is reflected to its negative.
The zeroth digit of any $v$ is, thus, either $e$ or $\ppar-e$. Unless $v=e$, this implies $0\in\Dset$.
\end{proof}

\begin{lemma}\label{lemma:prodzero}
For $e\in\eve^{<\ppar}$ and $v,w\in \block{e}{\ppar}\setminus\{e\}$ we have $\loopdown{}{v}\loopdown{}{w}=0$.
\end{lemma}

\begin{proof}
By \fullref{lemma:zerodigit} we have $0\in\Dset$ for any $v\in\block{e}{\ppar}$
with $v\neq e$. If $\Cset\cap \Cset[w]=\emptyset$, then we have
$\loopdown{}{v}\loopdown{}{w}=0$ trivially. Otherwise the product
$\loopdown{}{v}\loopdown{}{w}$ is supported on certain $z\in\Cset\cap\Cset[w]$,
but there it is a multiple of $\loopdown{2}{0}\pjw[w{-}1]=0$, see
\fullref{lemma:dualnumbers}.
\end{proof}

Next will be that the $\loopdown{}{v}$ are central, which needs:

\begin{lemma}\label{lemma:tech}
Let $v\in\block{e}{\ppar}$ and $i,k\in\Dset$ with $i\neq k$ and $i\in\Dset[{v[k]}]$. Then we have:
\begin{gather}
\label{eq:centralone}
\pjw[v{-}1]\loopdown{}{i}
\Up{k}\pjw[{v[k]{-}1}]
=\pjw[v{-}1]\Up{k} 
\loopdown{}{i}\pjw[{v[k]{-}1}]\neq 0,
\\
\label{eq:centraltwo}
\pjw[v{-}1]\loopdown{}{i}
\Up{i}\pjw[{v[i]{-}1}]
=\pjw[v{-}1]\Up{i} 
\loopdown{}{i}\pjw[{v[i]{-}1}]=0.
\end{gather}
In particular, a loop of type $\loopdown{}{i}$ can be transported along $\Up{k}$, and
thus, also along $\Down{k}$, due to symmetry. 
\end{lemma}

\begin{proof} 
We rewrite both sides of 
the desired equations and introduce notation as
follows. Here we allow the case $k=i$, for which we have to prove that both sides are zero.
\begin{align*}
\loopdown{}{i}\Up{k} 
\pjw[{v[k]{-}1}] 
&=  
(-1)^{a_{i}}a_{i}\Up{S_{i}}\Down{S_{i}}\Up{S_{k}}
\pjw[{v[k]{-}1}] 
\\
\Up{k} \loopdown{}{i}
\pjw[{v[k]{-}1}] 
&= 
(-1)^{b_{i}}b_{i} 
\Up{S_{k}} 
\Up{T_{i}}\Down{T_{i}}\pjw[{v[k]{-}1}],
\end{align*} 
where $S_{i}=\overline{\{i\}}$ 
and $S_{k}=\overline{\{k\}}$ are the
down-admissible hulls for $v$, 
$T_{i}:=\overline{\{i\}}$ is the
down-admissible hull for $v[k]$, 
and $b_{i}$ denotes the $i$th digit of
$v[k]$. In checking \eqref{eq:centralone} and \eqref{eq:centraltwo}, there are four cases to consider.
\begin{enumerate}[label=$\bullet$]

\setlength\itemsep{0.15cm}

\item If $S_{i}$ and $S_{k}$ are distant, then $T_{i}=S_{i}$ and $a_{i}=b_{i}$, and
\eqref{eq:centralone} follows from far-commutativity 
\fullref{theorem:main-tl-section}.(3).

\item If $S_{k}>S_{i}$ are adjacent, then we have $a_{i}=b_{i}$ and $S_{i}$ is
down-admissible for $v[k]$, which implies $S_{i}=T_{i}$. Now we distinguish two cases. Suppose $a_{k}\neq 1$, then we compute
\begin{align*}
\Up{S_{i}}\Down{S_{i}}\Up{S_{k}}\pjw[{v[k]{-}1}] 
&= \Up{S_{i}}\Up{S_{k}} \Up{S_{i}} \pjw[{v[k]{-}1}] 
\\
&= 
\funcg(a_{k}-1)\Up{S_{k}} \Down{S_{i}}\Up{S_{i}}\pjw[{v[k]{-}1}]
\\
&= 
\funcg(a_{k}-1)
\Up{S_{k}}\Big(\funcg(\ppar-a_{k}) 
\Up{S_{i}} \Down{S_{i}} 
+ 
\funcf(\ppar-a_{k})\Up{S_{k}}\Up{S_{i}}\Down{S_{i}}\Down{S_{k}}\Big) \pjw[{v[k]{-}1}] 
\\
&= 
\Up{S_{k}} \Up{S_{i}}\Down{S_{i}}\pjw[{v[k]{-}1}]  
=\Up{S_{k}} \Up{T_{i}} \Down{T_{i}}\pjw[{v[k]{-}1}],
\end{align*}
as the $\funcf$ term gets killed by containment 
\fullref{theorem:main-tl-section}.(2) and we have
$\funcg(a_{k}-1)\funcg(\ppar-a_{k})=1$, because $a_{k}\neq 1$. 

\noindent Now suppose $a_{k}=1$. Then we compute:
\begin{align*}
\Up{S_{i}}\Down{S_{i}}\Up{S_{k}}\pjw[{v[k]{-}1}] 
&= 
\Up{S_{i}}\Up{S_{k} \cup \{k-1\}} \Up{S_{i}\setminus\{k-1\}}  \pjw[{v[k]{-}1}] 
\\
&= 
\Up{S_{k}}\Up{S_{i}}\Down{\{k-1\}}\Up{S_{i}\setminus\{k-1\}}  \pjw[{v[k]{-}1}] 
\\
&= 
\Up{S_{k}}\Up{S_{i}} \Down{S_{i}}\pjw[{v[k]{-}1}],
\end{align*}
where we have used the overlap relation \fullref{theorem:main-tl-section}.(5).
In either case, we deduce \eqref{eq:centralone} since $a_{i}=b_{i}$. 

\item Suppose that $S_{i}>S_{k}$ are adjacent. Then we have $a_{i}-1=b_{i}$. By the
admissibility assumption we have $b_{i}\neq 0$, which implies $T_{i}=S_{i}$, and we
compute
\begin{align*}
\Up{S_{k}} \Up{T_{i}}\Down{T_{i}}\pjw[{v[k]{-}1}] 
&=
\Up{S_{k}} \Up{S_{i}}\Down{S_{i}}\pjw[{v[k]{-}1}]
\\ 
&= 
\funcg(a_{i}-1) \Up{S_{i}} \Down{S_{k}}\Down{S_{i}}\pjw[{v[k]{-}1}]
\\
&= 
\funcg(a_{i}-1) \Up{S_{i}}\Down{S_{i}}\Up{S_{k}}\pjw[{v[k]{-}1}].
\end{align*}
Here we observe $\funcg(a_{i}-1)= 
-\frac{a_{i}}{a_{i}-1}=\frac{(-1)^{a_{i}} a_{i}}{(-1)^{b_{i}}b_{i}}$, 
which verifies \eqref{eq:centralone}.

\item Finally, if $S_{i}=S_{k}$, then we have $\Up{S_{i}}\Down{S_{i}}\Up{S_{i}}\pjw[{v[k]{-}1}]=0$
by \eqref{eq:DUD}. For the other side of the equation we consider $\Up{S_{i}}\Up{T_{i}}\Down{T_{i}}\pjw[{v[k]{-}1}]$. 
If $T_{i}=\{i\}$, then $S_{i}\supset T_{i}$ and 
the expression is zero by the
containment relation \fullref{theorem:main-tl-section}.(2). Otherwise we necessarily have $S_{i}=\{i\}\subset T_{i}$
and we use the overlap relation \fullref{theorem:main-tl-section}.(5) to get
\begin{gather*} 
\Up{S_{i}}\Up{T_{i}}\Down{T_{i}}\pjw[{v[k]{-}1}]
= 
\Up{T_{i}\setminus S_{i}}\Up{S_{i}}\Down{S_{i}}\Down{T_{i}} \pjw[{v[k]{-}1}]
=0,
\end{gather*} 
where we have again used containment at the end. Hence, \eqref{eq:centraltwo} holds.
\end{enumerate}\vspace{-.6cm}
\end{proof}

\begin{lemma}\label{lemma:loopscentral} 
We have the following:
\begin{enumerate}[label=(\alph*)]

\setlength\itemsep{0.15cm}

\item If $z\in\aZ(\zigzagc_{e{-}1})$ 
satisfies $z\pjw[v{-}1]=c
\loopdown{}{v}\pjw[v{-}1]$ 
for some $v\in\block{e}{\ppar}\setminus\{e\}$ and $c\in\K$, then $z\pjw[w{-}1]$ also contains 
$\loopdown{}{v}\pjw[w{-}1]$ with
coefficient $c$ for any $w\in\Cset$.

\item For $v\in\block{e}{\ppar}\setminus\{e\}$ we have $\loopdown{}{v}\in\aZ(\zigzagc_{e{-}1})$. 
\end{enumerate}
\end{lemma}

\begin{proof}
For the first part, expand $\loopdown{}{v}$ into a product of basic loops and
observe that the $w\in\Cset$ are 
precisely the vertices to which
$\loopdown{}{v}$ can be transported using \fullref{lemma:tech}. In a central
element, these loops $\loopdown{}{v}$ therefore have to appear with
the same coefficient at any $w\in\Cset$.

For the second part, we write $v=\pbase{a_{j},\dots,a_{0}}{\ppar}$ and observe
that every arrow $\Up{k}$ or $\Down{k}$ in the quiver, for which we write $Y_{k}$,
is one of the following:
\begin{enumerate}[label=(\roman*)]

\setlength\itemsep{0.15cm}

\item $Y_{k}$ is not adjacent to any $w\in\Cset$, in which case it trivially commutes:
\begin{gather*}
Y_{k}\loopdown{}{v}=0,
\quad\loopdown{}{v}Y_{k}=0.
\end{gather*}

\item $Y_{k}$ is adjacent to an $w\in\Cset$, but not a generator for the equivalence
relation $\sim_{\Dset}$. In this case \fullref{lemma:charCv} implies $k\leq j$
and commutation follows from \eqref{eq:centraltwo}.

\item $Y_{k}$ is a generating 
arrow of the equivalence relation
$\sim_{\Dset}$, in which case 
it commutes with the loops of type $\loopdown{}{\Dset}$ by
\eqref{eq:centralone}.\qedhere

\end{enumerate}
\end{proof}

The following example shows that the assumption $i\in\Dset[{v[k]}]$ in \fullref{lemma:tech} is necessary.

\begin{example}\label{example:twelve}
Let $\ppar=3$. Then there is no element in $z\in\aZ(\zigzagc_{0})$ such that $z\pjw[12]=\Up{1}\Down{1}\pjw[12]$.
Indeed, we can compute 
\begin{gather*}
z\pjw[10]\Down{0}\pjw[12]=\Down{0}z\pjw[12] = \Down{0}\Up{1}\Down{1}\pjw[12]=\Up{1,0}\Down{1}\pjw[12] 
\end{gather*}
However, $\Up{1,0}\Down{1}\pjw[12]\neq z\pjw[10]\Down{0} \pjw[12]$ for any $z\pjw[10]$, as is easily verified since the endomorphisms
of $\pjw[10]$ are spanned by $\pjw[10]$ and $\Up{0}\Down{0}\pjw[10]$.

Neither is there $z\in\aZ(\zigzagc_{0})$ such that 
$z\pjw[16]=\Up{1}\Down{1}
\pjw[16]$, as such a loop, if 
central, could be transported to $\pjw[12]$.
\end{example}

Finally, we show that
$\loopdown{}{v}$, together with the unit, form a basis of the center.

\begin{lemma}\label{lemma:loopscentral-op}
The center $\aZ(\zigzagc_{e{-}1})$ has a basis given by the unit ({\cf}
\eqref{eq:unit}) and the $\loopdown{}{v}$ for $v\in\block{e}{\ppar}$ with $v\neq e$. 
\end{lemma}

\begin{proof}
To see linear independence, first note that no linear combination of
the (nilpotent) $\loopdown{}{v}$ can be the unit. Moreover, assuming that
$\sum_{i=1}^{r}a_{i}\loopdown{}{v_{i}}=0$ for $a_{i}\in\K$ and
$v_{i}<v_{j}$ for $i<j$, we can multiply this equation with $\pjw[v_{1}{-}1]$ and
get
$a_{1}\loopdown{}{\Dset[v_{1}]}\pjw[v_{1}{-}1]=a_{1}\loopdown{}{v_{1}}\pjw[v_{1}{-}1]=\sum_{i=1}^{r}a_{i}\loopdown{}{v_{i}}\pjw[v_{1}{-}1]=0$.
Hence, $a_{1}=0$, since $\loopdown{}{\Dset[v_{1}]}\pjw[v_{1}{-}1]$ is a basis
element in the corresponding endomorphism ring. We can repeat this process to
show that all $a_{i}$ are zero.

Suppose now we are given $z\in\aZ(\zigzagc_{e{-}1})$. Let $v\in\block{e}{\ppar}$ be
minimal with $z\pjw[v{-}1]\neq 0$. If $v=e$, then $z\pjw[v{-}1]=c\cdot\pjw[v{-}1]$ for some non-zero scalar $c\in\K$. In this case, we proceed with
the central element $z-c\cdot 1$ in place of $z$.

Now we assume that $v\neq e$ and claim that $z\pjw[v{-}1]=c\cdot
\loopdown{}{v}\pjw[v{-}1]$ 
for some non-zero scalar $c\in\K$ since otherwise we could find a smaller $v$. 

To see this, suppose that $z\pjw[v{-}1]$ contains a summand $d\cdot\loopdown{}{S} \pjw[v{-}1]$ and suppose there is an $i\in\Dset\setminus S$.
We write $S_i:=\overline{\{i\}}$ 
for the admissible hull and split
$S=S_{>i}\cup S_{<i}$ into subsets 
of indices greater and smaller than $i$
respectively. Then $z\Down{i}\pjw[v{-}1]=\Down{i}z\pjw[v{-}1]$ contains a summand
\begin{align*}
d\cdot\Down{S_i}\loopdown{}{S}\pjw[v{-}1] 
&= 
d\cdot\Down{S_i} \Up{S_{>i}}\Up{S_{<i}}\Down{S_{<i}}\Down{S_{>i}}\pjw[v{-}1]
\\
&=
\begin{cases}
d\cdot\Up{S_{>i}\cup S_i}\Up{S_{<i}}\Down{S_{<i}}\Down{S_{>i}}\pjw[v{-}1] 
&  
\text{if }S_{>i}\text{ is adjacent to }i,
\\
d\cdot \Up{S_{>i}}\Up{S_{<i}}\Down{S_{<i}}\Down{i}\Down{S_{>i}}\pjw[v{-}1] 
&
\text{if }S_{>i}\text{ is distant to }i,
\\
\end{cases}
\end{align*}
where we have used the adjacency 
relation \fullref{theorem:main-tl-section}.(4) in the first case, the
far-commutativity relation \fullref{theorem:main-tl-section}.(3) in the 
second case, and also that $\Down{S_i}$
commutes with $\Up{S_{<i}}\Down{S_{<i}}$, 
as verified in the proof of
\fullref{lemma:tech}. In either case the 
result is non-zero. A similar
argument now shows that this term 
can not be canceled by other summands of
$\Down{i}z\pjw[v{-}1]$. Thus, we have $z\Down{i} \pjw[v{-}1]\neq 0$, in
contradiction to the minimality of $v$.

Now define $z^{\prime}=z-c\loopdown{}{v}$. By
\fullref{lemma:loopscentral} we have $z^{\prime}\in\aZ(\zigzagc_{e{-}1})$ and by
\fullref{corollary:Cv} we know that $z^{\prime}$ is supported at vertices $w>v$. Now
proceed by induction.
\end{proof}

\section{The center of the tilting category}\label{section:center-tilt}

We now explain how $\aZ(\zigzagc)$ and $\cZ(\tilt)$ 
are related.

\subsection{Some general (and well-known) facts}\label{subsection:facts}

Let $\kk$ denote any commutative unital ring.
For an additive $\kk$-linear category $\catstuff{C}$ we let
$\catstuff{End}(\catstuff{C})$ denote its category of $\kk$-linear endofunctors
(which are thus automatically additive) and natural transformations. Within this category we have the
$\kk$-algebra of endomorphism of any endofunctor $\morstuff{F}$ denoted by
$\setstuff{End}_{\catstuff{End}(\catstuff{C})} (\morstuff{F})$. Recall the
following definition.

\begin{definition}
Let $\catstuff{C}$ be an additive $\kk$-linear category. Then 
its center $\cZ(\catstuff{C})$ is defined as
\begin{gather*}
\cZ(\catstuff{C})
=
\setstuff{End}_{\catstuff{End}(\catstuff{C})}
(\morstuff{Id}_{\catstuff{C}})
,
\end{gather*}
{\ie} the natural transformations of the identity functor $\morstuff{Id}_{\catstuff{C}}$
of $\catstuff{C}$.
\end{definition}

\begin{lemma}\label{lemma:cat-center}
We have the following.
\begin{enumerate}[label=(\alph*)]

\setlength\itemsep{0.15cm}	

\item The space $\cZ(\catstuff{C})$ is a commutative $\kk$-algebra.

\item If $\catstuff{C}$ and $\catstuff{D}$ are equivalent additive $\kk$-linear categories, then 
$\cZ(\catstuff{C})\cong\cZ(\catstuff{D})$
as commutative $\kk$-algebras.

\item If
$\catstuff{C}\cong\prod_{i\in I}\catstuff{C}_{i}$ as additive $\kk$-linear
categories with trivial hom spaces between $\catstuff{C}_{i}$ and
$\catstuff{C}_{j}$ unless $i=j$, then
$\cZ(\catstuff{C})\cong
{\textstyle\prod_{i\in I}}\cZ(\catstuff{C}_{i})$
as commutative $\kk$-algebras.

\end{enumerate}
\end{lemma}

\begin{proof}
The first claim is evident. For the second claim we observe that, for any additive $\kk$-linear full subcategory 
$\catstuff{D}\subset\catstuff{C}$, restriction defines 
a $\kk$-algebra homomorphism $r\colon\cZ(\catstuff{C})\to\cZ(\catstuff{D})$. If $\catstuff{C}$ and $\catstuff{D}$ are equivalent then one can check that $r$ is an isomorphism. The final claim can be proven {\muta} as the second claim.
\end{proof}

Let us consider the example $\prmod{\setstuff{A}}$
(projective right modules of some $\kk$-algebra $\setstuff{A}$).

\begin{lemma}\label{lemma:cat-center2}
Let $\setstuff{A}$ be a $\kk$-algebra.
\begin{enumerate}[label=(\alph*)]

\setlength\itemsep{0.15cm}	

\item We have an inclusion
$\aZ(\setstuff{A})\hookrightarrow \cZ(\prmod{\setstuff{A}})$ of commutative $\kk$-algebras.

\item If $\setstuff{A}$ is unital, then we have an isomorphism
$\aZ(\setstuff{A})\xrightarrow{\cong}\cZ(\prmod{\setstuff{A}})$ of commutative $\kk$-algebras.

\end{enumerate}
\end{lemma}

\begin{proof}
We claim that the map 
\begin{gather*}
\aZ(\setstuff{A})\ni c \mapsto
\chi_{c}:=
\big\{
\chi_{\obstuff{P}}\colon\obstuff{P}\to\obstuff{P},
m\mapsto mc
\mid
\obstuff{P}\in\prmod{\setstuff{A}}
\big\}
\end{gather*}
gives the required inclusion or isomorphism, respectively.

Indeed, it is easy to see that $\chi_{c}$ is actually in 
$\cZ(\prmod{\setstuff{A}})$ and that 
this map is a well-defined inclusion of $\kk$-algebras.

To prove surjectivity under the assumption of unitality of $\setstuff{A}$, 
given $\chi\in\cZ(\prmod{\setstuff{A}})$, 
its value on the projective right $\setstuff{A}$-module 
$\setstuff{A}$ is by definition a map of right $\setstuff{A}$-modules $\chi_{\setstuff{A}}\colon\setstuff{A}\to\setstuff{A}$. Let 
$c:=\chi_{\setstuff{A}}(1)\in\setstuff{A}$, which actually belongs to $\aZ(\setstuff{A})$. Next, for any 
$\obstuff{P}\in\prmod{\setstuff{A}}$ and any $m\in\obstuff{P}$ 
there is a unique $f\in\Hom_{\prmod{\setstuff{A}}}(\setstuff{A},\obstuff{P})$ with $f(1)=m$. Using this and naturality, we get
\begin{gather*}
\chi_{\obstuff{P}}(m)=
\chi_{\obstuff{P}}\big(f(1)\big)
=f\big(\chi_{\setstuff{A}}(1)\big)
=f(c)=mc,
\end{gather*}
which proves surjectivity.
\end{proof}

\subsection{The tilting center}\label{subsection:tilt-center}

Let us modify \fullref{lemma:cat-center2}, which we need to do since $\zigzag$ is non-unital:

\begin{lemma}\label{lemma:cat-center3}
We have an isomorphism of commutative $\K$-algebras given by
\begin{gather*}
\aZ(\zigzagc)\xrightarrow{\cong}\cZ(\prmod{\zigzag}),
\quad
\aZ(\zigzagc)\ni c\mapsto
\chi_{c}:=
\big\{
\chi_{\obstuff{P}}\colon\obstuff{P}\to\obstuff{P},
m\mapsto mc
\mid
\obstuff{P}\in\prmod{\zigzag}
\big\}
.
\end{gather*}
\end{lemma}

\begin{proof}
Defining $c:=\chi_{\zigzagc}(\sum_{v\in\N}\pjw[v{-}1])$, 
the argument is the same as in \fullref{lemma:cat-center2}.
\end{proof}

By the well-understood block decomposition
\begin{gather*}
\tilt
=
{\textstyle\bigoplus_{e\in\eve}}\,
\tilt_{e{-}1},
\quad
\tilt_{e{-}1}=
\big\{
\obstuff{T}(v-1)\mid
v\in\block{e}{\ppar}
\big\},
\end{gather*}
where hom spaces between $\tilt_{e{-}1}$ and 
$\tilt_{e^{\prime}{-}1}$ are trivial unless $e=e^{\prime}$.

\begin{definition} 
For $v\in\block{e}{\ppar}\setminus\{e\}$ we define the natural transformation
\begin{gather}\label{eq:tiltnattrafo}
\chi_{v}:=
\big\{
\chi_{\obstuff{T}}\colon\obstuff{T}\to\obstuff{T},
\chi_{\obstuff{T}}=\loopdown{}{v}
\mid
\obstuff{T}\in\tilt_{e{-}1}
\big\},
\end{gather}
where $\loopdown{}{v}$ was defined in \fullref{definition:Lv}.
\end{definition}

Thus, we can finally prove \fullref{theorem:main}, which is 
a consequence of:

\begin{theorem}\label{theorem:center-category}
For $e\in\eve^{<\ppar}$ we have an $\K$-algebra isomorphism
\begin{gather*}
\cZ(\tilt_{e{-}1})\xrightarrow{\cong}
{\K}\big[X_{v}\mid v\in\block{e}{\ppar}\setminus\{e\}\big]
\Big/
\big\langle X_{v}X_{w}\mid v,w\in\block{e}{\ppar}\setminus\{e\}\big\rangle,
\quad
\chi_{v}\mapsto X_{v}.
\end{gather*}
\end{theorem}

\begin{proof}
By \fullref{theorem:main-tl-section} and \fullref{lemma:cat-center}.(b)
we need to compute $\cZ(\prmod{\zigzag_{e-1}})$.
This in turn, by \fullref{lemma:cat-center3}, reduces 
to compute the algebra 
center of $\zigzagc_{e-1}$, which is 
\fullref{theorem:center-algebra}.
\end{proof}

\subsection{The center and Donkin's tensor product theorem}\label{subsection:tensor}
We expect that the central elements $\loopdown{}{v}$ admit an interpretation via Donkin’s tensor product 
theorem, which, in our notation, takes the following form.

\begin{propositionqed}\label{proposition:donkin}(See \cite[Proposition 2.1]{Do-tilting-alg-groups}.)
For $v=\pbase{a_{j},\dots,a_{0}}{\ppar}$
we have an isomorphism of (left) $\SLtwo$-modules
\begin{gather*}
\obstuff{T}(v-1)
\cong
\obstuff{T}(a_{j}-1)^{j}
\otimes
\big({\textstyle\bigotimes_{i=0}^{j-1}}\,
\obstuff{T}(a_{i}+\ppar-1)^{i}\big)
,
\end{gather*}
where the superscripts indicate Frobenius twist.
\end{propositionqed}

\begin{lemma}\label{lemma:donkin}
Let $v=\pbase{a_{j},\dots,a_{0}}{\ppar}$ and $v^{\prime}=\pbase{1,
a_{j-1},\dots,a_{0}}{\ppar}$. For all
$w=\pbase{b_{k},\dots,b_{j},b_{j-1},\dots,b_{0}}{\ppar}\in\Cset$, define
$w^{\prime}=\pbase{b_{k},\dots,b_{j}}{\ppar}$. We have an isomorphism of (left)
$\SLtwo$-modules
\begin{gather*}
\obstuff{T}(w-1)
\cong
\obstuff{T}(w^{\prime}-1)^j
\otimes
\obstuff{T}(v^{\prime}-1)
.
\end{gather*}
\end{lemma}

\begin{proof}
First, note that \fullref{lemma:charCv}
gives $w=\pbase{b_{k},\dots,b_{j+1},b_{j},
a_{j-1},\dots,a_{0}}{\ppar}$.
Thus, applying \fullref{proposition:donkin} proves the claim.
\end{proof}

\begin{example}\label{example:donkin}
As before, let $\ppar=7$, $v=\pbase{3,1,6,5,0,5,6}{7}$ with
$v^{\prime}=\pbase{1,1,6,5,0,5,6}{7}$. Then
$w=\pbase{1,4,1,6,5,0,5,6}{7}\in\Cset$ and $w^{\prime}=\pbase{1,4}{7}$.
\fullref{proposition:donkin} and \fullref{lemma:donkin} give
\begin{gather*}
\obstuff{T}(v^{\prime}-1)
\cong
\obstuff{T}(7)^{5}
\otimes
\obstuff{T}(12)^{4}
\otimes
\obstuff{T}(11)^{3}
\otimes
\obstuff{T}(6)^{2}
\otimes
\obstuff{T}(11)^{1}
\otimes
\obstuff{T}(12)^{0}
,
\\
\obstuff{T}(v-1)
\cong
\obstuff{T}(2)^{6}
\otimes
\obstuff{T}(v^{\prime}-1)
,
\\
\obstuff{T}(w-1)
\cong
\obstuff{T}(w^{\prime}-1)^{6}
\otimes
\obstuff{T}(v^{\prime}-1)
\cong
\obstuff{T}(0)^{7}
\otimes
\obstuff{T}(10)^{6}
\otimes
\obstuff{T}(v^{\prime}-1)
.
\end{gather*}
\end{example}

In the situation of \fullref{lemma:donkin} where
\begin{gather*}
\obstuff{T}(w-1)
\cong
\obstuff{T}(w^{\prime}-1)^{j}
\otimes
\obstuff{T}(v^{\prime}-1)
\end{gather*}
it seems plausible that $\loopdown{}{v}\pjw[w{-}1] =
c\cdot\morstuff{id}_{\obstuff{T}(w^{\prime}-1)^j}\otimes
\loopdown{}{v^{\prime}}$ for some scalar $c\in\K$. One approach to proving this
statement involves giving a diagrammatic interpretation of Frobenius twists in
the Temperley--Lieb category. Alternatively, it would also be interesting to
have a proof that such elements are indeed central and span the center, which
does not use the explicit presentation of $\tilt$, {\eg} using the abstract basis
given in \cite{AnStTu-cellular-tilting}.

\subsection{Blocks via the center}\label{subsection:blocks}

Let us discuss a classical application of the center of a category, mimicking the well-understood case of category $\mathcal{O}$. Note that $\cZ(\catstuff{C})$ acts naturally on all
objects $\obstuff{X}$ of $\catstuff{C}$ via
\begin{gather*}
\chi\acts\morstuff{f}:=\morstuff{f}\circ\chi_{\obstuff{X}}
=\chi_{\obstuff{X}}\circ\morstuff{f},
\quad\chi\in\cZ(\catstuff{C}),\morstuff{f}\in\End_{\catstuff{C}}(\obstuff{X}).
\end{gather*}
In particular, for any character of the center
$\chi_{s}\colon\cZ(\catstuff{C})\to\K$,
we can fiber the category $\catstuff{C}$ 
by defining full subcategories
$\catstuff{C}_{s}\subset\catstuff{C}$ consisting of all objects 
where $\cZ(\catstuff{C})$ acts with character $\chi_{s}$.

By \fullref{theorem:main}, the $\catstuff{C}_{s}$ for $\catstuff{C}=\catstuff{Tilt}$ are exactly the blocks $\tilt_{e{-}1}$, which follows because 
each $\K$-algebra $\cZ(\tilt_{e{-}1})$, which has only one non-nilpotent basis element, has exactly one simple.

\subsection{The Casimir element}\label{subsection:casimir}

Let $C\colon A\to\K$ denote the Casimir element of $\SLtwo$, which is an element
in its distribution algebra, {\ie} a map from the coordinate ring $A$ of
$\SLtwo$ to $\K$. Recall that $C$ acts on every $\SLtwo$ module $\obstuff{M}$
using the coaction $\Delta_{\obstuff{M}}$ of $A$ on $\obstuff{M}$:
\begin{gather*}
C
\colon
\obstuff{M}\xrightarrow{\Delta_{\obstuff{M}}}\obstuff{M}\otimes_{\K}A
\xrightarrow{1\otimes C}\obstuff{M}\otimes_{\K}\K
\xrightarrow{\cong}\obstuff{M}.
\end{gather*}
Let $\obstuff{L}(v-1)$ denote the simple $\SLtwo$ module of highest weight
$v-1$. We choose $C = (h+1)^{2} + 4fe$ and see that $C$ acts on
$\obstuff{L}(v-1)$ as $v^2\morstuff{id}$. The proof of \fullref{lemma:zerodigit}
implies that $C$ acts on $\obstuff{L}(v-1)$ as multiplication by
$v^{2}=a_{0}^{2}= (\ppar-a_{0})^2=e^{2}\in\F$ for every $v\in\block{e}{\ppar}$. Since
all $L(v-1)$ that appear in composition series of objects of $\tilt_{e-1}$ have
$v\in\block{e}{\ppar}$, we see that $C$ acts by this scaling on all objects in
$\tilt_{e-1}$. Hence, in the blocks $\aZ(\zigzagc_{e{-}1})$ of the Ringel dual we have
$C=e^{2}\cdot 1$.

\section{Two other cases}\label{section:othercases}

\begin{remark}
The two theorems below can be proven, {\muta}, 
as for $\catstuff{Tilt}$, but the computations 
are much simpler. Thus, we decided to keep the proofs short.
\end{remark}

\subsection{A variation: the quantum case (generation \texorpdfstring{$1$}{1})}\label{subsection:qgroup}

In the quantum group case the same calculations, using the Ringel dual computed
in \cite{AnTu-tilting}, work, but are much simpler, and so is the result. Let
$\tilt^{q}$ denote the quantum analog of $\tilt$ with quantum parameter $q$, where we use Lusztig's divided power quantum group.
Throughout this part, \cite{AnTu-tilting} serves as our reference for statements
about $\tilt^{q}$. However, we use the notation from this paper. In particular,
we write $\pbase{a_{j},\dots,a_{0}}{k}:=\sum_{i=0}^{j}a_{i} k^{i}=v$ for the
$k$-adic expansion with digits $0\leq a_i<k$.

We will work over a field $\C(q)$ and distinguish the following cases:
\begin{enumerate}[label=$\bullet$]

\setlength\itemsep{0.15cm}

\item The quantum parameter $q$ is a formal parameter, or $q\in \C$ is $q=\pm 1$ or not a root of unity. Here we set $k=\infty$.

\item The quantum parameter $q\in\C$ is a root of unity $q\neq \pm 1$. Here $k$ denotes the order of $q^{2}$.

\end{enumerate}

We also let
\begin{gather}\label{eq:nattrafos}
\chi_{v}:=
\big\{
\chi_{\obstuff{T}}\colon\obstuff{T}\to\obstuff{T},
\chi_{\obstuff{T}}=\loopdown{\{0\}}{v}
\mid
\obstuff{T}\in\tilt
\big\},
\quad
\delta_{v}:=
\big\{
\delta_{\obstuff{T}}\colon\obstuff{T}\to\obstuff{T},
\delta_{\obstuff{T}}=\delta_{v}
\mid
\obstuff{T}\in\tilt
\big\},
\end{gather}
where $\loopdown{\{0\}}{v}$ and $\delta_{v}$ are 
the maps that act 
as non-zero only on summands 
$\obstuff{T}(v-1)$, on which they are a
head-to-socle map and the identity, respectively.

\begin{theorem}\label{theorem:qgroup}	
We have the following.
\begin{enumerate}[label=(\alph*)]

\setlength\itemsep{0.15cm}

\item For $k=\infty$ there is an 
isomorphism of $\C(q)$-algebras
\begin{gather*}
\cZ(\tilt^{q})\xrightarrow{\cong}
{\textstyle\prod_{\N}}\,\C(q),
\quad
\delta_{v}\mapsto 1_{v}.
\end{gather*}

\item For the root of unity case there is an 
isomorphism of $\C$-algebras
\begin{gather}\label{eq:qtiltcenter}
\begin{gathered}
\cZ(\tilt^{q})
\xrightarrow{\cong}
{\textstyle\bigoplus_{i=-1}^{k-2}}\,
\obstuff{X}(i),
\quad
\obstuff{X}(i)
=
\begin{cases}
{\textstyle\prod_{\N}}\,\C 
&\text{if }i=-1,
\\
\C[X_{v}\mid v\in\N]\Big/\langle X_{v}X_{w}\mid v,w\in\N\rangle 
& \text{otherwise},
\end{cases}
\end{gathered}
\end{gather}
where $\chi_{w}\mapsto X_{v}$, 
and $\delta_{x}\mapsto 1_{v}$ for $w,x\in\N$ as 
explained in the proof.
\end{enumerate}
\end{theorem}

\begin{proof}
\noindent(a). This claim is clear as 
$\tilt^{q}$ is semisimple and its 
simples are indexed by $\N$. 

\noindent(b). For the root of unity case 
we recall that we have equivalences of additive $\K$-linear categories
\begin{gather*}
\tilt^{q}\cong
\tilt^{q}_{St}
\oplus
\tilt^{q}_{0}
\oplus
\dots
\oplus
\tilt^{q}_{k{-}2},
\\
\tilt^{q}_{0}
\cong
\dots
\cong
\tilt^{q}_{k{-}2}
\cong
\prmod{\zigzag^{q}},
\end{gather*}
where $\tilt^{q}_{St}$ is semisimple with 
simple objects being the indecomposable quantum tilting modules whose highest weight $v$ satisfies 
$a_{0}=0$, which is the 
first case in \eqref{eq:qtiltcenter}.

Moreover, the Ringel dual $\zigzag^{q}$ 
of $\tilt^{q}_{0}$ is a zigzag algebra with a boundary condition on the vertex set $\N$, {\ie}
\begin{gather*}
\begin{tikzcd}[ampersand replacement=\&,column sep=3em]  
(v_{0} -1)
\ar[r,"\Up{\{0\}}",black,yshift=0.1cm]
\&
(v_{1} -1)
\ar[r,"\Up{\{0\}}",black,yshift=0.1cm]
\ar[l,"\Down{\{0\}}",orchid,yshift=-0.1cm]
\&
(v_{2} -1)
\ar[r,"\Up{\{0\}}",black,yshift=0.1cm]
\ar[l,"\Down{\{0\}}",orchid,yshift=-0.1cm]
\&
(v_{3} -1)
\ar[r,"\Up{\{0\}}",black,yshift=0.1cm]
\ar[l,"\Down{\{0\}}",orchid,yshift=-0.1cm]
\&
\dots
\ar[l,"\Down{\{0\}}",orchid,yshift=-0.1cm]
\end{tikzcd}
,
\end{gather*}
where $v_0=1$ and $v_{i+1}=v_{i}(0)$ for $i\geq 0$, 
subject to the relations
\begin{gather*}
\Down{\{0\}}\Down{\{0\}}\pjw[v{-}1]=0,
\Up{\{0\}}\Up{\{0\}}\pjw[v{-}1]=0,
\;
\Down{\{0\}}\Up{\{0\}}\pjw[v{-}1]
=\Up{\{0\}}\Down{\{0\}}\pjw[v{-}1]
\text{ for }v\neq 1,
\;
\Down{\{0\}}\Up{\{0\}}\pjw[0]=0.
\end{gather*}
(Note the boundary relation $\Down{\{0\}}\Up{\{0\}}\pjw[0]=0$.)
Formally, the algebra $\zigzag^{q}$ is generated by $\pjw[v{-}1]$ for $v\in\N$ being of the form 
$v=1(0)\dots(0)$, 
and elements $\Down{\{0\}}\pjw[v{-}1]$ and $\Up{\{0\}}\pjw[v{-}1]$, modulo the relations shown above, which are the analogs of the relations in \fullref{theorem:main-tl-section}.

Observe that we have central elements in $\zigzag^{q}$ of the form
\begin{gather*}
\loopdown{\{0\}}{v}=\Down{\{0\}}\Up{\{0\}}\pjw[v{-}1]=\Up{\{0\}}\Down{\{0\}}\pjw[v{-}1],\text{ for }v\neq 1.
\end{gather*}
These are central since they are annihilated 
by any element except their idempotent. Moreover, no other element is central, implying the second case in \eqref{eq:qtiltcenter}, using arguments as in \fullref{subsection:facts} and 
the quantum analog of \fullref{lemma:cat-center3}. 
In particular, $\loopdown{\{0\}}{v}$ corresponds 
to the natural transformation in \eqref{eq:nattrafos} 
and satisfies the relations of $X_{v}$ in 
$\C[X_{v}\mid v\in\N]\big/\langle X_{v}X_{w}\mid v,w\in\N\rangle$.
\end{proof}

\subsection{A variation: the \texorpdfstring{$G_{g}T$}{GrT} case for \texorpdfstring{$g=1,2$}{g=1,2} (generations \texorpdfstring{$1,2$}{1,2})}\label{subsection:ggt}

Recall that $\pbase{a_{j},\dots,a_{0}}{\ppar}=\sum_{i=0}^{j}a_{i}
\ppar^{i}=v$.
Using the same methods, 
in the case of projective $G_{g}T$-modules for $\SLtwo$
and $g=1,2$ one can also calculate the center 
of the corresponding additive $\K$-linear category $\gtmod{g}$. 
The corresponding Ringel duals were computed in \cite{An-tilting}, which is also our reference for statements about $\gtmod{g}$. 
Again, the resulting algebras 
are much simpler than for $\catstuff{Tilt}$.

Before we can state the theorem, let us define natural transformations $\loopdown{\{0\}}{v}$ and $\delta_{v}$ 
{\muta} as in \eqref{eq:nattrafos}. We also need
\begin{gather*}
\chi_{v}^{X}:=
\big\{
\chi_{\obstuff{P}}^{X}\colon\obstuff{P}\to\obstuff{P},
\chi_{\obstuff{P}}^{X}=\loopdown{\{0\}}{v}\loopdown{\{1\}}{v}
\mid
\obstuff{P}\in\gtmod{2}
\big\},
\\
\chi_{x}^{Y}:=
\big\{
\chi_{\obstuff{P}}^{Y}\colon\obstuff{P}\to\obstuff{P},
\chi_{\obstuff{P}}^{Y}=\loopdown{\{0\}}{x}
\mid
\obstuff{P}\in\gtmod{2}
\big\},
\end{gather*}
where both, $\loopdown{\{0\}}{v}\loopdown{\{1\}}{v}$ 
and $\loopdown{\{0\}}{x}$, are non-zero only on corresponding indecomposable projective $G_{2}T$-modules $\obstuff{P}(v{-}1)$ and $\obstuff{P}(x{-}1)$ for $v\in\Z\setminus\ppar\Z$ 
and $x\in\ppar\Z$, where they are the head-to-socle maps.
We further need an analog of \eqref{eq:tiltnattrafo}, namely
\begin{gather*}
\chi_{i}^{Z}:=
\big\{
\chi_{\obstuff{P}}^{Z}\colon\obstuff{P}\to\obstuff{P},
\chi_{\obstuff{P}}^{Z}=\loopdown{}{i}
\mid
\obstuff{P}\in\gtmod{2}
\big\},
\end{gather*}
where $\loopdown{}{i}$ is only zero
on all indecomposable summands $\obstuff{P}(j{-}1)$ 
for $j\in r(i)$ (the row of $i$, 
{\cf} \eqref{eq:p-is-5}). On such $\obstuff{P}(j{-}1)$ 
it is a map that factors through the highest 
weights of its horizontal neighbors (on the grid as in \eqref{eq:p-is-5}), which is unique up to scalars.

\begin{theorem}\label{theorem:grt}
We have the following.

\begin{enumerate}[label=(\alph*)]

\setlength\itemsep{0.15cm}

\item For $g=1$ we have $\K$-algebra isomorphisms
\begin{gather}\label{eq:g1tcenter}
\begin{gathered}
\cZ\big(\gtmod{1}\big)
\xrightarrow{\cong}
{\textstyle\bigoplus_{i=-1}^{k-2}}\,
\obstuff{X}(i),
\;
\obstuff{X}(i)
=
\begin{cases}
{\textstyle\prod_{\Z}}\,\K 
&\text{if }i=-1,
\\
\K[X_{v}\mid v\in\Z]\Big/\langle X_{v}X_{w}\mid v,w\in\Z\rangle 
& \text{otherwise},
\end{cases}
\end{gathered}
\end{gather}
where $\chi_{w}\mapsto X_{v}$, 
and $\delta_{x}\mapsto 1_{v}$ for $w,x\in\Z$ as 
explained in the proof.

\item For $g=2$ we have $\K$-algebra isomorphisms
\begin{gather}\label{eq:g2tcenter}
\begin{gathered}
\cZ\big(\gtmod{2}\big)
\xrightarrow{\cong}
{\textstyle\prod_{i=0}^{\ppar-2}}\,
\obstuff{X}(i),
\\
\obstuff{X}(i)
\cong
\begin{cases}
\cZ\big(\gtmod{1}\big) 
&\text{if }i=-1,
\\
\K[X_{v},Y_{x},Z_{i}\mid 
v,x\in\Z,
i\in\eve^{<\ppar}]\Big/
\setstuff{I}
& \text{otherwise},
\end{cases}
\end{gathered}
\end{gather}
where
\begin{gather*}
\setstuff{I}=
\Bigg\langle 
\begin{gathered}
X_{v}X_{w},X_{v}Y_{y},X_{v}Z_{j},
\\[-6pt]
Y_{x}X_{w},Y_{x}Y_{y},Y_{x}Z_{j},
\\[-6pt]
Z_{i}X_{w},Z_{i}Y_{y},Z_{i}Z_{j},
\end{gathered}
\,\Bigg\vert\,
\begin{gathered}
v,w\in\Z,
\\[-6pt]
x,y\in\Z,
\\[-6pt]
i,j\in\eve^{<\ppar}
\end{gathered}
\Bigg\rangle
,
\end{gather*}
and $\loopdown{X}{w}\mapsto X_{v}$, 
and $\loopdown{Y}{y}\mapsto Y_{x}$ for $w,y\in\Z$, 
and $\loopdown{Z}{j}\mapsto Z_{i}$ 
for $j\in\eve^{<\ppar}$ as 
explained in the proof.
\end{enumerate}

\end{theorem}

\begin{proof}
\noindent(a). Very similar as in the quantum case we 
have a decomposition of additive $\K$-linear categories
\begin{gather*}
\gtmod{1}\cong
\gtmod{1}_{St}
\oplus
\gtmod{1}_{0}
\oplus
\dots
\oplus
\gtmod{1}_{\ppar{-}2},
\\
\gtmod{1}_{0}
\cong
\dots
\cong
\gtmod{1}_{\ppar{-}2}
\cong
\prmod{\zigzag^{G_{1}T}},
\end{gather*}
where $\gtmod{2}_{St}$ is semisimple with 
simple objects being projective $G_{1}T$-modules 
indexed as in the first case in \eqref{eq:g1tcenter}.
The Ringel dual quiver algebra $\zigzag^{G_{1}T}$ of $\gtmod{1}_{0}$
in this case is a zigzag algebra on the vertex set $\Z$, {\ie}
\begin{gather*}
\begin{tikzcd}[ampersand replacement=\&,column sep=3em]
\dots
\ar[r,"\Up{\{0\}}",black,yshift=0.1cm]
\&
(v_{-1}-1)
\ar[r,"\Up{\{0\}}",black,yshift=0.1cm]
\ar[l,"\Down{\{0\}}",orchid,yshift=-0.1cm]
\& 
(v_{0}-1)
\ar[r,"\Up{\{0\}}",black,yshift=0.1cm]
\ar[l,"\Down{\{0\}}",orchid,yshift=-0.1cm]
\&
(v_{1}-1)
\ar[r,"\Up{\{0\}}",black,yshift=0.1cm]
\ar[l,"\Down{\{0\}}",orchid,yshift=-0.1cm]
\&
(v_{2}-1)
\ar[r,"\Up{\{0\}}",black,yshift=0.1cm]
\ar[l,"\Down{\{0\}}",orchid,yshift=-0.1cm]
\&
\dots
\ar[l,"\Down{\{0\}}",orchid,yshift=-0.1cm]
\end{tikzcd}
,
\end{gather*} 
where $v_{0}=1$ and $v_{i+1}=v_{i}(0)$ for $i\in\N[0]$, $v_{-i}=-v_{i}+2$ for $i$
even, and $v_{-i}=-v_{i}+2\ppar-2$ for $i$ odd. These are the numbers with
$\ppar$-adic expansion $\pbase{n_1,a_0}{\ppar}$ where $a_{0}\in
\{0,\dots,\ppar-1\}$ and $n_1\in \Z$, which can be reached from $1$ by
successive upward or downward reflection in the zeroth digit. The generators
$\Down{\{0\}}$ and $\Up{\{0\}}$ are subject to the relations
\begin{gather*}
\Down{\{0\}}\Down{\{0\}}\pjw[v{-}1]=0,
\Up{\{0\}}\Up{\{0\}}\pjw[v{-}1]=0,
\;
\Down{\{0\}}\Up{\{0\}}\pjw[v{-}1]
=\Up{\{0\}}\Down{\{0\}}\pjw[v{-}1]
.
\end{gather*}
Formally, the algebra $\zigzag^{G_{1}T}$ is generated by $\pjw[v{-}1]$ for $v\in\Z$ being of the form 
$v=\pm\big(1(0)\dots(0)\big)$, 
and elements $\Down{\{0\}}\pjw[v{-}1]$ and $\Up{\{0\}}\pjw[v{-}1]$. The relations are the ones above
together with the usual idempotent relations. Thus, the situation is analogous to the quantum case.
In particular, the same arguments, {\muta}, as for the quantum group give the second case in \eqref{eq:g1tcenter}.

\noindent(c). We start by recalling that
\begin{gather*}
\gtmod{2}\cong
\gtmod{2}_{St}
\oplus
\gtmod{2}_{0}
\oplus
\dots
\oplus
\gtmod{2}_{\ppar{-}2},
\\
\gtmod{2}_{St}
\cong
\gtmod{1},
\quad
\gtmod{2}_{0}
\cong
\dots
\cong
\gtmod{2}_{\ppar{-}2}
\cong
\prmod{\zigzag^{G_{2}T}},
\end{gather*}
where $\gtmod{2}_{St}$ is no longer semisimple, but rather equivalent to
$\gtmod{1}$. This gives us the first summand in \eqref{eq:g2tcenter}. The Ringel
dual quiver algebra $\zigzag^{G_{2}T}$ still has vertex set $\Z$, but arranged
on a grid, {\eg} if $\ppar=5$, then, as explained in \cite[Section
6.3]{An-tilting}:
\begin{gather}\label{eq:p-is-5}
\begin{tikzcd}[ampersand replacement=\&,column sep=3.75em,row sep=3.75em]
\& \vdots\ar[d,"\Down{\{1\}}",swap,xshift=-0.1cm,tomato] \& \vdots\ar[d,"\Down{\{1\}}",swap,xshift=-0.1cm,tomato] \& \vdots\ar[d,"\Down{\{1\}}",swap,xshift=-0.1cm,tomato] \& \vdots\ar[d,"\Down{\{1\}}",swap,xshift=-0.1cm,tomato] \&
\\
\& w_{9}\ar[r,"\Down{\{0\}}",swap,yshift=-0.1cm,orchid]\ar[d,"\Down{\{1\}}",swap,xshift=-0.1cm,tomato]\ar[u,"\Up{\{1\}}",swap,xshift=0.1cm,spinach] \& w_{8}\ar[l,"\Up{\{0\}}",swap,yshift=0.1cm,black]\ar[r,"\Down{\{0\}}",swap,yshift=-0.1cm,orchid]\ar[d,"\Down{\{1\}}",swap,xshift=-0.1cm,tomato]\ar[u,"\Up{\{1\}}",swap,xshift=0.1cm,spinach] \& w_{7}\ar[l,"\Up{\{0\}}",swap,yshift=0.1cm,black]\ar[r,"\Down{\{0\}}",swap,yshift=-0.1cm,orchid]\ar[d,"\Down{\{1\}}",swap,xshift=-0.1cm,tomato]\ar[u,"\Up{\{1\}}",swap,xshift=0.1cm,spinach] \& w_{6}\ar[l,"\Up{\{0\}}",swap,yshift=0.1cm,black]\ar[r,"\Down{\{0\}}",swap,yshift=-0.1cm,orchid]\ar[d,"\Down{\{1\}}",swap,xshift=-0.1cm,tomato]\ar[u,"\Up{\{1\}}",swap,xshift=0.1cm,spinach] \& w_{5}\ar[l,"\Up{\{0\}}",swap,yshift=0.1cm,black]
\\
w_{0}\ar[r,"\Up{\{0\}}",yshift=0.1cm,black] \& w_{1}\ar[r,"\Up{\{0\}}",yshift=0.1cm,black]\ar[l,"\Down{\{0\}}",yshift=-0.1cm,orchid]\ar[d,"\Down{\{1\}}",swap,xshift=-0.1cm,tomato]\ar[u,"\Up{\{1\}}",swap,xshift=0.1cm,spinach] \& w_{2}\ar[r,"\Up{\{0\}}",yshift=0.1cm,black]\ar[l,"\Down{\{0\}}",yshift=-0.1cm,orchid]\ar[d,"\Down{\{1\}}",swap,xshift=-0.1cm,tomato]\ar[u,"\Up{\{1\}}",swap,xshift=0.1cm,spinach] \& w_{3}\ar[r,"\Up{\{0\}}",yshift=0.1cm,black]\ar[l,"\Down{\{0\}}",yshift=-0.1cm,orchid]\ar[d,"\Down{\{1\}}",swap,xshift=-0.1cm,tomato]\ar[u,"\Up{\{1\}}",swap,xshift=0.1cm,spinach] \& w_{4}\ar[l,"\Down{\{0\}}",yshift=-0.1cm,orchid]\ar[d,"\Down{\{1\}}",swap,xshift=-0.1cm,tomato]\ar[u,"\Up{\{1\}}",swap,xshift=0.1cm,spinach] \&
\\
\& w_{-1}\ar[r,"\Down{\{0\}}",swap,yshift=-0.1cm,orchid]\ar[d,"\Down{\{1\}}",swap,xshift=-0.1cm,tomato]\ar[u,"\Up{\{1\}}",swap,xshift=0.1cm,spinach] \& w_{-2}\ar[l,"\Up{\{0\}}",swap,yshift=0.1cm,black]\ar[r,"\Down{\{0\}}",swap,yshift=-0.1cm,orchid]\ar[d,"\Down{\{1\}}",swap,xshift=-0.1cm,tomato]\ar[u,"\Up{\{1\}}",swap,xshift=0.1cm,spinach] \& w_{-3}\ar[l,"\Up{\{0\}}",swap,yshift=0.1cm,black]\ar[r,"\Down{\{0\}}",swap,yshift=-0.1cm,orchid]\ar[d,"\Down{\{1\}}",swap,xshift=-0.1cm,tomato]\ar[u,"\Up{\{1\}}",swap,xshift=0.1cm,spinach] \& w_{-4}\ar[l,"\Up{\{0\}}",swap,yshift=0.1cm,black]\ar[r,"\Down{\{0\}}",swap,yshift=-0.1cm,orchid]\ar[d,"\Down{\{1\}}",swap,xshift=-0.1cm,tomato]\ar[u,"\Up{\{1\}}",swap,xshift=0.1cm,spinach] \& w_{-5}\ar[l,"\Up{\{0\}}",swap,yshift=0.1cm,black]
\\
w_{-10}\ar[r,"\Up{\{0\}}",yshift=0.1cm,black] \& w_{-9}\ar[r,"\Up{\{0\}}",yshift=0.1cm,black]\ar[l,"\Down{\{0\}}",yshift=-0.1cm,orchid]\ar[d,"\Down{\{1\}}",swap,xshift=-0.1cm,tomato]\ar[u,"\Up{\{1\}}",swap,xshift=0.1cm,spinach] \& w_{-8}\ar[r,"\Up{\{0\}}",yshift=0.1cm,black]\ar[l,"\Down{\{0\}}",yshift=-0.1cm,orchid]\ar[d,"\Down{\{1\}}",swap,xshift=-0.1cm,tomato]\ar[u,"\Up{\{1\}}",swap,xshift=0.1cm,spinach] \& w_{-7}\ar[r,"\Up{\{0\}}",yshift=0.1cm,black]\ar[l,"\Down{\{0\}}",yshift=-0.1cm,orchid]\ar[d,"\Down{\{1\}}",swap,xshift=-0.1cm,tomato]\ar[u,"\Up{\{1\}}",swap,xshift=0.1cm,spinach] \& w_{-6}\ar[l,"\Down{\{0\}}",yshift=-0.1cm,orchid]\ar[d,"\Down{\{1\}}",swap,xshift=-0.1cm,tomato]\ar[u,"\Up{\{1\}}",swap,xshift=0.1cm,spinach] \&
\\
\& \vdots\ar[u,"\Up{\{1\}}",swap,xshift=0.1cm,spinach] \& \vdots\ar[u,"\Up{\{1\}}",swap,xshift=0.1cm,spinach] \& \vdots\ar[u,"\Up{\{1\}}",swap,xshift=0.1cm,spinach] \& \vdots\ar[u,"\Up{\{1\}}",swap,xshift=0.1cm,spinach] \&
\end{tikzcd}
.
\end{gather}
The indexing of the positively labeled vertices $w_i=v_i-1$ hereby works as
follows. $v_{0}=\pbase{1}{5}(=1)$, and each horizontal step is $(0)$, {\eg}
$v_{1}=\pbase{1}{5}(0)=\pbase{2,-1}{5}=\pbase{1,4}{5}(=9)$, while each
vertical step is $(1)$, {\eg}
$v_{9}=\pbase{1,4}{5}(1)=\pbase{2,-1,4}{5}(=49)$. Moreover, $w_{-i}=-w_{i}$, if
$i\geq 0$ is even, $w_{-i}=-w_{i}+2\ppar-4$, if $i>0$ is odd. The $v_i$ are
exactly the integers with $\ppar$-adic expansion $\pbase{n_2,a_1,a_0}{\ppar}$
where $a_0,a_1\in \{0,\dots, \ppar-1\}$ and $n_2\in \Z$, which can be reached
from $1$ by reflection in the first two digits.

Precisely, the algebra $\zigzag^{G_{2}T}$ is generated by $\pjw[v{-}1]$ for
$v\in\Z$ being as above, and elements $\Down{\{0\}}\pjw[v{-}1]$ and
$\Up{\{0\}}\pjw[v{-}1]$ as well as $\Down{\{1\}}\pjw[v{-}1]$ and
$\Up{\{1\}}\pjw[v{-}1]$. The relations are such that each column and each row is
a zigzag algebra, and all squares commute, {\ie}:

\begin{enumerate}[label=(\arabic*)]

\setlength\itemsep{0.15cm}

\item Each row is a zigzag algebra, {\cf} \fullref{theorem:main-tl-section}.(2) and (6), {\ie}
\begin{gather*}
\Up{\{0\}}\Up{\{0\}}\pjw[v{-}1]=
0=\Down{\{0\}}\Down{\{0\}}\pjw[v{-}1],
\quad
\Down{\{0\}}\Up{\{0\}}\pjw[v{-}1]
=\Up{\{0\}}\Down{\{0\}}\pjw[v{-}1]
.
\end{gather*}
(No boundary condition.)

\item Each column is a zigzag algebra, {\cf} \fullref{theorem:main-tl-section}.(2) and (6), {\ie}
\begin{gather*}
\Up{\{1\}}\Up{\{1\}}\pjw[v{-}1]=
0=\Down{\{1\}}\Down{\{1\}}\pjw[v{-}1],
\quad
\Down{\{1\}}\Up{\{1\}}\pjw[v{-}1]
=\Up{\{1\}}\Down{\{1\}}\pjw[v{-}1]
.
\end{gather*}

\item All squares commute, {\cf} \fullref{theorem:main-tl-section}.(4), {\ie}
\begin{gather*}
\Down{\{1\}}\Up{\{0\}}\pjw[v{-}1]
=
\Down{\{0\}}\Down{\{1\}}\pjw[v{-}1],
\quad
\Down{\{0\}}\Up{\{1\}}\pjw[v{-}1]
=
\Up{\{1\}}\Up{\{0\}}\pjw[v{-}1],
\\
\Down{\{1\}}\Down{\{0\}}\pjw[v{-}1]
=
\Down{\{0\}}\Up{\{1\}}\pjw[v{-}1],
\quad
\Down{\{0\}}\Down{\{1\}}\pjw[v{-}1]
=
\Down{\{1\}}\Up{\{0\}}\pjw[v{-}1]
.
\end{gather*}
These hold only for complete squares, {\ie} 
\begin{gather*}
\begin{tikzcd}[ampersand replacement=\&,column sep=3em,row sep=3em]  
\&
w_{9}
\\
w_{0}
\ar[r,"\Up{\{0\}}",yshift=0.1cm,black]
\&
w_{1}
\ar[u,"\Up{\{1\}}",swap,xshift=0.1cm,spinach]
\end{tikzcd}
,
\end{gather*}
as in \eqref{eq:p-is-5}, does not satisfy any relation, and is in particular, not zero.

\end{enumerate}

Hence, we get row and column loops
\begin{gather*}
\loopy{i}:=\loopy{\{0\}}\pjw[v_{i}{-}1]=\Down{\{0\}}\Up{\{0\}}\pjw[v_{i}{-}1]=\Up{\{0\}}\Down{\{0\}}\pjw[v_{i}{-}1],
\\
\loopyy{i}:=\loopy{\{1\}}\pjw[v_{i}{-}1]=\Down{\{1\}}\Up{\{1\}}\pjw[v_{i}{-}1]=\Up{\{1\}}\Down{\{1\}}\pjw[v_{i}{-}1].
\end{gather*}
(Here, as in \eqref{eq:p-is-5}, only one expression for 
$\loopy{\{0\}}\pjw[v_{i}{-}1]$ makes sense if $\ppar|i$. The above is just a shorthand notation.)
The relations imply that
\begin{gather}\label{eq:some-relations}
\pjw[v_{i}{-}1]
\zigzag^{G_{2}T}
\pjw[v_{i}{-}1]
\cong
\begin{cases}
\K[\loopy{i}]\Big/\langle\loopy{i}^{2}\rangle
&\text{if }\ppar|i,
\\
\K[\loopy{i},\loopyy{i}]\Big/\langle\loopy{i}^{2},(\loopyy{i})^{2}\rangle
&\text{otherwise}.
\end{cases}
\end{gather}
Note that these have bases $\{\pjw[v_{i}{-}1],\loopy{i}\}$ and 
$\{\pjw[v_{i}{-}1],\loopy{i},\loopyy{i},\loopy{i}\loopyy{i}\}$,respectively.
Also, the loops $\loopy{i}$, for $\ppar|i$,
and $\loopy{i}\loopyy{i}$, otherwise, are central since 
only the idempotent $\pjw[v_{i}{-}1]$ does not annihilate them.

But there are more central elements.
To define them, let $c(j)$ denote the set of indexes 
of the $j$th column, read left to right. 
For example, $c(2)=\{\dots,-8,-2,2,8,\dots\}$ for $\ppar=5$, {\cf} \eqref{eq:p-is-5}.
Then we sum row loops over their column $r(j)$, {\ie} 
\begin{gather*}
\loopy{c(j)}={\textstyle\sum_{i\in c(j)}}\,\loopy{i}\in
\zigzagc^{G_{2}T}.
\end{gather*}
We have $\loopy{c(j)}\in\aZ(\zigzagc^{G_{2}T})$: That $\loopyy{c(j)}$ commutes with all idempotents is clear. 
Moreover, each vertical arrow annihilates $\loopyy{c(j)}$ from both sides 
and each horizontal arrow transports a loop $\loopy{i}$ to its neighbor, illustrated as
\begin{gather*}
\begin{tikzcd}[ampersand replacement=\&,column sep=3em,row sep=3em]  
i \& j
\\
k\ar[r,yshift=0.1cm,black]\ar[u,xshift=0.1cm,spinach] \& l\ar[l,yshift=-0.1cm,orchid]
\end{tikzcd}
=
\begin{tikzcd}[ampersand replacement=\&,column sep=3em,row sep=3em]  
i \& j\ar[l,yshift=-0.1cm,orchid]
\\
k\ar[r,yshift=0.1cm,black] \& l\ar[u,xshift=0.1cm,spinach]
\end{tikzcd}
=
\begin{tikzcd}[ampersand replacement=\&,column sep=3em,row sep=3em]  
i\ar[r,yshift=0.1cm,black] \& j\ar[l,yshift=-0.1cm,orchid]
\\
k\ar[u,xshift=0.1cm,spinach] \& l
\end{tikzcd}
.
\end{gather*} 
Furthermore, a bit more thought (using arguments as in \fullref{section:center-alg}) proves that the central elements which we have identified, {\ie}
\begin{gather*}
\loopy{i},\text{ for }\ppar|i,
\quad
\loopy{i}\loopyy{i},\text{ otherwise},
\quad
\loopyy{r(j)},\text{ for }j\in\{1,\dots,\ppar-1\},
\end{gather*}
form a basis of $\aZ(\zigzagc^{G_{2}T})$, and it remains to calculate the relations among these. 
We already know the relations among the elements $\loopy{i}$ and $\loopy{i}\loopyy{i}$, see \eqref{eq:some-relations}. Further, we have
\begin{gather*}
\loopy{c(j)}\loopy{c(k)}=0,
\quad
\loopy{c(j)}\loopy{i}=0=\loopy{i}\loopy{c(j)},\text{ for }\ppar|i,
\quad
\loopy{c(j)}\loopy{i}\loopyy{i}=0=\loopy{i}\loopyy{i}\loopy{c(j)}
.
\end{gather*}
This concludes the proof.
\end{proof}

\end{document}